\documentclass{amsart}
\usepackage{amsmath,amssymb,amsthm,amscd}
\usepackage[matrix,arrow, curve]{xy}
\numberwithin{equation}{section}
\usepackage{color}
\usepackage[hypertexnames=false,linkcolor=black, colorlinks=true, citecolor=black, filecolor=black,urlcolor=black]{hyperref}

\theoremstyle{plain}
\newtheorem{Thm}{Theorem}[section]
\newtheorem{Lem}[Thm]{Lemma}
\newtheorem{Prop}[Thm]{Proposition}
\newtheorem{Cor}[Thm]{Corollary}

\newtheorem{ConstrIntro}{Construction}
\newtheorem{NotIntro}[ConstrIntro]{Notation}
\newtheorem{ThmIntro}[ConstrIntro]{Theorem}

\theoremstyle{definition}
\newtheorem{Def}[Thm]{Definition}
\newtheorem{Rk}[Thm]{Remark}
\newtheorem{Ex}[Thm]{Example}

\newtheorem{Not_num}[Thm]{Notation}
\newtheorem{Obs}[Thm]{Observation}
\newtheorem{Constr}[Thm]{Construction}

\theoremstyle{remark}

\newcommand{\qo}{quasi-ordinary }
\newcommand{\wrt}{with respect to }

\newcommand{\gqz}{{\geq 0}}
\newcommand{\nupoly}[4]{\Delta^{#1} (\, #2; \, #3; \, #4 \,)}
\newcommand{\cnupoly}[3]{\Delta^{#1} (\, #2; \, #3\,)}

\newcommand{\0}{{\bf 0}}
\renewcommand{\a}{{\bf a}}
\renewcommand{\b}{{\bf b}}

\newcommand{\q}{{\bf q}}

\renewcommand{\u}{{\bf u}}
\newcommand{\U}{{\bf U}}

\newcommand{\x}{{\bf x}}
\newcommand{\X}{{\bf X}}
\newcommand{\y}{{\bf y}}
\newcommand{\z}{{\bf z}}

\newcommand{\M}{{\bf M}}
\newcommand{\N}{{\bf N}}
\newcommand{\PL}{{\bf P}}
\newcommand{\Alpha}{{\boldsymbol\alpha}} 
\newcommand{\Beta}{{\boldsymbol\beta}}

\newcommand{\field}{ K }
\newcommand{\coeffforf}{\rho}
\newcommand{\coeffforfmodified}{\mu}
\newcommand{\coeffforoverweight}{\rho}
\newcommand{\coeffforoverweightplus}{\mu}
\newcommand{\theinvariant}{\kappa}
\newcommand{\GOOD}{\infty}
\newcommand{\BAD}{-1}
\newcommand{\lastentry}{\xi}

\newcommand{\indexorder}{{poly}}
\newcommand{\weight}{ W}

\newcommand{\IA}{\mathbb{ A }}
\newcommand{\IC}{\mathbb{ C }}

\newcommand{\IQ}{\mathbb{ Q }}
\newcommand{\IR}{\mathbb{ R }}
\newcommand{\IZ}{\mathbb{ Z }}

\newcommand{\mfX}{{\mathfrak{X}}}

\begin{document}

\thispagestyle{empty}

\title
{
A polyhedral characterization of quasi-ordinary singularities
}

\author{Hussein Mourtada}
\thanks{The first author is  partially supported by the ANR-12-JS01-0002-01 SUSI}
\address{
Hussein Mourtada\\
Equipe G\'eom\'etrie et Dynamique \\
Institut Math\'ematique de Jussieu-Paris Rive Gauche\\
Universit\'e Paris 7 \\
B\^atiment Sophie Germain, case 7012\\
75205 Paris Cedex 13, France}
\email{hussein.mourtada@imj-prg.fr}

\author{Bernd Schober}
\thanks{The second author is supported by Research Fellowships of the Deutsche Forschungsgemeinschaft (SCHO 1595/1-1 and SCHO 1595/2-1).}
\address{Bernd Schober\\
Johannes Gutenberg-Universit\"at Mainz, Fachbereich 08, Staudingerweg 9, 55099 Mainz, Germany}
\curraddr{
	Institut f\"ur Algebraische Geometrie\\
	Leibniz Universit\"at Hannover\\
	Welfengarten 1\\
	30167 Hannover\\
	Germany}
\email{schober@math.uni-hannover.de}


\begin{abstract}
	Given an irreducible hypersurface singularity of dimension $d$ 
	(defined by a polynomial $f\in \field[[ \x ]][z]$) 
	and the projection to the affine space defined by $\field[[ \x ]]$, we construct
	an invariant which detects whether the singularity is quasi-ordinary with respect to the projection. 
	The construction uses a weighted version of Hironaka's characteristic polyhedron and
	successive embeddings of the singularity in affine spaces of higher dimensions.  
	When $ f $ is quasi-ordinary, our invariant determines the semigroup of the singularity and hence it encodes the embedded topology of the singularity $ \{ f = 0 \} $ in a neighbourhood of the origin when $ \field = \IC$
	and $ f $ is complex analytic; 
	moreover, we explain the relation between the construction and the approximate roots.
\end{abstract}

\maketitle

\section*{Introduction}
\label{Intro}

Let $\field$ be an algebraically closed field of characteristic $0$ and let us denote by $\field[[ \x ]]$ the power series ring $\field[[x_1,\ldots,x_d]],~d \in \IZ_+.$ 
The first objects considered 
in this paper are quasi-ordinary polynomials: 
a Weierstrass polynomial
$$ 
	f = f (\x,z) = z^n + f_1(\x) z^{n-1} + \ldots + f_{n-1}(\x) z + f_{n} (\x) \in \field [[\x]][z] 
$$
satisfying $f(\0, 0 )=0$ is said to be quasi-ordinary if its discriminant as a polynomial in $z$ is a monomial up to
multiplication by a unit in $\field[[ \x ]].$ 
A celebrated theorem by Abhyankar and Jung states that the roots of such a polynomial sit in $\field[[x_1^{\frac{1}{n}},\ldots,x_d^{\frac{1}{n}}]]$, see \cite{Jung} and \cite{Abhy_ram_algfct} (see also \cite{Lu}, \cite{PR} and see \cite{Cu2} for a generalization of this theorem).

From a different point of view, one can consider the formal quasi-ordinary germ $(V=\{f=0\},0);$
then the singularities of $(V,0)$ are intimately related to the roots of $f.$ 
When $d=1,$ i.e., when $V$ is a plane curve, then the Newton algorithm for determining the roots of $f$ gives also a resolution of singularities of $V,$ \cite{Cu}. 
When $d>1,$ this assertion makes sense thanks to the notion
of 
``characteristic exponents'' 
introduced
by Lipman \cite{LipmanThesis}. 
These invariants are extracted from the roots of $f,$ knowing the fact that the latter belong to  $\field[[x_1^{\frac{1}{n}},\ldots,x_d^{\frac{1}{n}}]],$
and it was proved by Gau that they determine the topological type of the singularity $(V,0)$ (when $ \field = \IC $
and $ f $ is complex analytic), see \cite{G}. 
Lipman also asked how one can construct an embedded resolution of singularities of 
$(V,0)\subset (\IA^d,0)$ from the characteristic exponents or equivalently from the roots of $f.$  
There exist many approaches to this question, e.g.~\cite{V}, \cite{GP1}, \cite{BMc}, \cite{CM}.  

This leads us to the other important objects for this paper, the invariants of resolutions of singularities. 
In particular, the following two approaches to prove the existence of an embedded resolution of singularities are crucial in our context:
The first approach is to construct an invariant which takes values in a ``well ordered set'' and to prove that there exists a finite sequence of blowing ups which makes this invariant strictly decrease.
Such an invariant should of course detect regularity, but should also not be too sophisticated in order to be able to follow its changes after blowing up.
The work of the second author suggests that this type of invariants is very much related to polyhedral invariants, namely to Hironaka's characteristic polyhedra \cite{H}, \cite{BerndBM}, \cite{CScompl}. 

The other approach is to resolve singularities by one toric morphism. 
This is not always possible if we do not change the ambient space. 
So the second approach is about finding an embedding in a higher dimensional affine space in such a way that one can resolve the singularities by one toric morphisms, \cite{T}, , \cite{GP1}, \cite{M1}, \cite{M2}, \cite{LMR}, \cite{Tevelev}.

In this paper, using a mixture of these two approaches, we build an invariant which detects whether $(V,0)$ is a quasi-ordinary singularity.
We first will introduce a weighted version of Hironaka's characteristic polyhedron and build our invariant from the weighted Hironaka polyhedra of successive embeddings of our singularity in  higher dimensional affine spaces. Let us give some details about this construction.

For $ c,d \in \IZ_+ ,$ let $\weight$ be a map $ \weight: \IZ^d_\gqz \to \IQ_\gqz^c$ which is the restriction of a linear map $\IQ^d \to \IQ^c.$
The data $ \weight $ is equivalent to the data of a $c \times d $ matrix that we also denote by $\weight.$
This map should be thought as a weight map on the monomials of $ \field [[ \x ]] $, where we assign to $ x_i $ the weight $ \weight ( e_i )\in\IQ_\gqz^c $ of the $ i $--th unit vector.
A special case is the identity $ \weight_0 : \IZ^d_\gqz \to \IQ^d_\gqz $ determined by the unit matrix $ Id_d.$  

Let $ f = \sum\limits_{\a, b} \coeffforf_{\a, b } \, \x^\a \, z^b \in \field [[\x]] [z] $ be a Weierstrass polynomial of degree $ n $. 
We associate with the projection given by the inclusion 
$\field[[ \x ]] \hookrightarrow \field [[\x ]][z] / \langle f \rangle,$ and with the weight $\weight,$ a polyhedron
$ \nupoly{\weight}{f}{\x}{z} $ which is defined to be
the smallest convex subset of $ \IR^c_\gqz $ containing all the points of the set
	$$
		\left\{ \,
			\frac{ \weight ( \a ) }{ n - b } + \IR^c_\gqz \; \mid\;  \coeffforf_{\a, b } \neq 0 \, \wedge \,  b < n \,
			\right\}\,.
	$$
 We minimize $ \nupoly{\weight}{f}{\x}{z} $ (for the inclusion) with respect to all changes of variables that respect the inclusion
 $\field[[ \x ]] \hookrightarrow \field [[\x ]][z] / \langle f \rangle$ 
(see section 2).

The polyhedron $ \nupoly{\weight}{f}{\x}{z} $ \wrt $ \weight$ is closely connected to Hironaka's characteristic polyhedron \cite{H}.
In fact, once it is minimal \wrt the choice of $ z $, it is the image of Hironaka's characteristic polyhedron under $ \weight $
{\label{CHANGE:01}considered as a map from $ \IR^d  $ to $ \IR^c $}
. 

With these polyhedra we define the notion of $ \nu $-\qo polynomials \wrt $ \weight $:
$ f $ is called a $ \nu $-\qo polynomial \wrt $ \weight $ 
{(and \wrt the above inclusion given by the $ ( \x) $)
if there exists some $ u_0 := z + h(\x) $ such that the polyhedron 
$ \nupoly{\weight}{f}{\x}{u_0} $ is either empty or has exactly one vertex $ v $ that cannot be eliminated
by a change of variable in $ u_0 $.} 
This generalizes Hironaka's notion of $ \nu $-\qo polynomials which are $ \nu $-\qo polynomials \wrt $ \weight_0.$

If $ f $ is $ \nu $-\qo \wrt $ W $ then the initial form at the unique vertex $ v $, i.e., the sum of those monomials determining
$ v $,
$$	
		F_{v, \weight }  = in_{v, \weight } (f) = U_0^{ n } - \sum_{\frac{\weight ( \a ) }{ n - b } =v} \coeffforf_{\a,b}
		 \X^{\a} U_0^b \,,
$$	
will be of particular interest for us.
If $ f $ is irreducible and $ \weight = \weight_0 $ then it is known that
$	
		F_{v, \weight_0 }  = ( U_0^{ m } - \coeffforf \X^{\a} )^{e} ,
$	
for certain $ m, e \in \IZ_+ $, $ m \cdot e = n $, $ \coeffforf \in \field^{\times} $ and $ \a \in \IZ^d_\gqz $ with $\frac{\a}{m} = v $ and $ \gcd(\a,m) = 1 $, see \cite{Evelia_Pedro_Compositio} resp.~\cite{ACLM_Factor}, or \cite{GR_Bernd} for a recent generalization to arbitrary fields.

\medskip

The construction of our invariant $ \theinvariant ( f; \x; z) $ goes as follows (for a detailed and precise version, see Construction \ref{The_Construction}).
If $f$ is not  $ \nu $-\qo then set  $\theinvariant ( f; \x; z) :=(\BAD).$ If $f$ is
$ \nu $-\qo then we set the first component of $\theinvariant ( f; \x; z)$ to be $v_1,$ which is the unique vertex  of {the minimal polyhedron $ \nupoly{\weight_0}{f}{\x}{u_0} $, $ u_0 = z + h_0(\x) $,} and let 
	$	
		F_{v_1, \weight_0 } = ( {U_0^{ n_1 }} - \coeffforf_1 \X^{\a_1} )^{e_1 } ,
	$	
	where $ \rho_1 \in \field^\times $, $ \frac{\a_1}{n_1} = v_1 $,
	and $ \gcd(\a_1,n_1) = 1 $.

	{\label{CHANGE:03}We extend $ \weight_0 $ to a linear map $ \weight_1 : \IZ^{ d+ 1}_\gqz \to \IQ_\gqz^d $ on $ \field [[ \x ]][u_0] $ by assigning
	$
		\weight_1 ( u_0) := v_1.
	$
	Let $  f^{(1)}  $ be the transform $ f $ which is obtained by exchanging $u_0^{ n_1 }$ in $f$ by $ z_1 + \coeffforf_1 \x^{\a_1}; $ 
	therefore we get that $ \ f^{(1)}  \in \field[[\x]][u_0]_{<n_1}[z_1], $ 
	i.e., $ u_0 $ appears to the power at most $ n_1 - 1 .$}
	Note that $  f^{(1)} $ is of degree $ e_1 = \frac{n}{n_1} < n $ in $ z_1 $.
	This strict inequality provides that our construction will be finite.
	
	Let
	\begin{equation}
	\label{eq:umini}
		u_1 := z_1 + h_1 ( \x, u_0 )	
	\end{equation}	
	be the change in $ z_1 $ such that the polyhedron
	$ \nupoly{\weight_1}{  f^{(1)}  }{\x, u_0}{u_1 } $
	becomes minimal (note that cleaning is this setting becomes slightly more subtle). 
	
	If $ f^{(1)}$ is not $ \nu $-\qo \wrt $ \weight_1 $ (for $ ( \x, u_0 ) $) then we set  
	$$
		\theinvariant ( f; \x ;z ) := ( \, v_1,\BAD).
	$$	
	If $ \nupoly{\weight_1}{ f^{(1)} } {\x, u_0}{u_1 } $ is empty then we put
	$$
		\theinvariant ( f; \x ;z ) := ( \, v_1,\GOOD).
	$$
	Otherwise, we denote by $ v_2 $ the unique vertex of $ \nupoly{\weight_1}{  f^{(1)}  }{\x, u_0}{u_1 } $ and we consider the  transform $f^{(2)}$ with respect to a change of variable suggested by a suitable initial form  $ F^{(1)}_{v_2, \weight_1 } $ of  $f^{(1)}$ at $ v_2.$ And we repeat the process till obtaining an empty polyhedron or a polyhedron with more than one vertex. After finitely many steps, say $ g \in \IZ_+ $, the construction ends and we define
	$$
		\theinvariant ( f; \x ; z ) := ( \, v_1, \, \ldots, \, v_g,\, \xi \,) \, ,
	$$ 
	where $ v_i $ is the vertex of the polyhedron in the $ i $--th step, $ \xi = \GOOD $ if the 
	polyhedron in the $ (g+1) $--th step is empty, and $ \xi = \BAD $ 
	else.	

The main result of this article is

\begin{ThmIntro}
\label{ThmIntro}
	Let $ f \in \field[[ \x ]][z ] $ be an irreducible $ \nu $-\qo polynomial of degree $ n $.
	Then $ f $ is a \qo polynomial if and only if the last entry of $ \theinvariant ( f; \x ;z ) $  is $ \GOOD $.
\end{ThmIntro}	
	
The key ingredient for the proof is:
if the last entry is $ \GOOD, $ then
our construction provides that,
after adding the variables $u_1,\ldots,u_g$,
the variety $V=\{f=0\}$ embedded in $ \IA_\field^{d+g+1} $ is the generic fiber of an ``overweight deformation'' of a toric variety. 
This is the monomial variety associated with the $v_i$ 
which is the closure of the orbit $(s_1^n,\ldots,s_d^n,S^{ n\cdot v_{1} },\ldots,S^{n\cdot v_{g} }),$ where $S=(s_1,\ldots,s_d)$ and 
$S^{(\alpha_1,\ldots,\alpha_d)}=s_1^{\alpha_1}\cdots s_d^{\alpha_d}.$ 
Hence we have a parametrization of the toric variety.
Using the overweight deformation 
we can lift 
the parametrization of the toric variety $ \mfX_0 $ to one
of $V = \{ f=0 \} $ embedded in $ \IA_\field^{d+g+1} $ 
after studying the equations of the space of solutions of $V$ in $(\field[[S]])^{d+g+1}$. 
This yields the roots of $ f $ and hence its discriminant \wrt $ z$ which happens to be a monomial times a unit in $\field[[ \x ]].$

{\label{CHANGE:05}On the other hand, if $ f $ is \qo then a direct computation of $ \theinvariant (f;\x;z) $ using the expression of $f$ in terms of ``approximate roots", following \cite{GP1}, gives that the last entry of  $ \theinvariant (f;\x;z) $ must be infinite, this is Proposition \ref{sdir}.}

As an other application of our invariant, we prove that when $f$ is \qo  $(\, v_1, \, v_2, \, \ldots, \, v_g \, ) $ determines a
system of generators of the semi-group of \qo hypersurface $ \{f=0\} $ and thus its topology when $ \field = \IC$
and $ f $ complex analytic.\\

We point out here that this paper is also inspired by \cite{ACLM_Newton} where they give a characterization of quasi-ordinary singularities using Newton trees. Our first attempt to prove our theorem was by using the theorem in \cite{ACLM_Newton}. \\ 

It is important to indicate that our invariant is not defined using the roots of $f$ and that the use of the roots is the classical way to define the generators of the semi-group \cite{GP2}. Actually, this was one of our primary motivations: 
{\em how to determine the invariants of a \qo singularity from its defining equation ?} 
When answering this question, we are searching to generalize the known invariants of quasi-ordinary singularities
to singularities which are defined by more general polynomials, for which the shape of the roots is unknown or very difficult to handle, in contrary to the \qo case.\\

The structure of the paper is as follows: in the first section we recall some basic facts about quasi-ordinary singularities. Section $2$  is devoted to weighted Hironaka's characteristic polyhedra. In section $3,$ we introduce the invariant
$ \theinvariant (f;\x;z) $ and show that its last entry is $\GOOD$ when $f$ is quasi-ordinary. Section $4$ is the last section and is devoted to the proof  of the other direction: if the last entry of $ \theinvariant (f;\x;z) $ is $\GOOD,$ $f$ is quasi-ordinary.    

\medskip

We shall use the notation in bold letters for the tuples $\x=(x_1,\ldots,x_d)$, and for $\x^\Alpha =x_1^{\alpha_1}\cdots x_d^{\alpha_d}$.
Throughout the whole paper the product order plays an important role. 
Therefore 
let us recall its definition.
		
\begin{NotIntro}	
\label{Not:poly_ordering}
	For $ v, w \in \IQ^d_\gqz $ we have:
	$$
		\begin{array}{lcl}
			v \geq_\indexorder w  	& :\Leftrightarrow & v \in w + \IQ^d_\gqz. \\[3pt]
			v >_\indexorder w  		& :\Leftrightarrow & v \geq_\indexorder w \;\; \wedge \;\; v \neq w.
		
		\end{array}
	$$
	Note that $ v \not\geq_\indexorder w $ does not imply  $ v <_\indexorder w $.
	Further, $ v =_\indexorder w $ (by which we mean $v \geq_\indexorder w  $ and $ v \leq_\indexorder w $)  is equivalent to $ v = w $.
\end{NotIntro}

%
%
%
%
%
%
%
%
%
%
%
%
%
%
%
%
%

\medskip

{\em Acknowledgments:}
The authors would like to thank Evelia Garcia-Barroso, Vincent Cossart, Dale Cutkosky, Adam Parusinski, Guillaume Rond, and Bernard Teissier for stimulating discussions and helpful comments.
Furthermore, they thank the anonymous referee for many important comments and corrections.

%
%
%
%
%
%
%
%
%
%
%
%
%
%
%
%
%

\medskip

\section{Quasi-ordinary singularities}

An equidimensional germ $ ( V , 0) $ of dimension $ d $ is {\it quasi-ordinary} 
if there exists a finite projection $ \pi : (V, 0) \rightarrow (\IA_\field^d,0) $
whose branch locus is a normal crossing divisor. 
If $ ( V, 0) $ is a hypersurface, $ (V, 0) \subset ( \IA_\field^{d+1},0) $, then $ V $ is
defined by a single equation $ f = 0$,  where  $f \in \field [[ x_1, \, \ldots, \, x_d ]] [z] $ is 
a Weierstrass polynomial 
whose discriminant with respect to $z$ is of the form
$\Delta_z f = x_1^{\delta_1}\cdots x_d^{\delta_d} \epsilon,$ where $\epsilon $ is a unit in  $\field [[ x_1,\ldots,x_d]]$  and
$(\delta_1, \, \ldots, \,\delta_d) \in \IZ^d_{\geqslant 0}$. 
In these coordinates the projection $\pi$ is 
induced by the inclusion
$$
\field [[ x_1,\ldots,x_d]] \hookrightarrow \field [[ x_1,\ldots,x_d]][ z]/ \langle f \rangle.
$$

\begin{Def} 
	Let $f\in \field[[\x]][z]$ be a polynomial in $z $. 
	\begin{enumerate} 
		\item 
		The polynomial $f$ is said to be {\em quasi-ordinary} if its discriminant as a polynomial in $ z $ is of the form $ \Delta_z f = x_1^{\delta_1}\cdots x_d^{\delta_d} \epsilon,$ where $\epsilon $ is a unit in  $\field [[ \x ]]$.

\smallskip 

		\item 
		The projection induced by the morphism $\field[[\x]]\longrightarrow \field[[\x]][z]/\langle f \rangle $ is said
		to be a {\em quasi-ordinary projection} if there exists a change of variables respecting  the projection (i.e., changes of the type $ x_i\mapsto x_i +h_i(\x) $ 
		 with $ h_i \in \field[[x]]$)
		such that the expression of $ f $ in the new variables is a quasi-ordinary polynomial.
	\end{enumerate}
\end{Def} 

\noindent 
{Note: Changes in $ z $ do not affect the discriminant of $ f $ as a polynomial in $ z $, hence we consider in (2) only changes in $ (\x) $.}

\smallskip

From now on, we assume that the hypersurface $(V,0)$ is analytically irreducible, i.e., $ f $ is irreducible in $ \field[[ \x]] [z].$ 
A crucial building {block} in the theory of quasi-ordinary singularities is

\begin{Thm}[Abhyankar-Jung Theorem]
Let $f\in \field [[\x]][z]$ be a quasi-ordinary polynomial of degree $n$ in $z.$
The roots of $f$ are fractional power series.
More precisely, if $ f $ is irreducible, they belong to the ring $\field [[x_1^{1/n}, \, \ldots,\, x_d^{1/n}]].$
\end{Thm}

Here and in the whole article, we implicitly mean that $ f $ is a {\it Weierstrass} polynomial of degree $ n $ in $ z$ if we write $ f $ is of degree $ n $ in $ z $.
Sometimes we even omit the reference to $ z $ if it is clear from the context.

\medskip

Let $\zeta^{(i)}, i \in \{ 1,\,\ldots,\,n \} $, be the roots of $f.$ 
The difference $\zeta^{(i)}-\zeta^{(j)}$ of two different roots divides the discriminant of $f$ in the 
ring $\field[[x_1^{1/n},\,\ldots,\,x_d^{1/n}]].$ 
Therefore 
$$
	\zeta^{(i)}-\zeta^{(j)}=\x^{\lambda_{ij}} \cdot \epsilon_{ij},
$$ 
where $ \epsilon_{ij} $ is a unit in $\field[[x_1^{1/n},\,\ldots,\,x_d^{1/n}]].$ 
These exponents have been introduced by Lipman in 
\cite{LipmanThesis}, and have been applied in \cite{L} or \cite{Zariski}, for example.
It follows from  Proposition 1.3 in \cite{G} that the exponents $ \lambda_{ij} $ are well ordered \wrt the product ordering $ \leq_\indexorder $ and so we name 
them 
$$
	\lambda_1 \, <_\indexorder \, \lambda_2 \, <_\indexorder \, \ldots \, <_\indexorder \, \lambda_g \,,
$$
and we call them the {\it characteristic exponents}.

We can then define the lattices $M_0 :=\IZ^{d} $ and $M_i := M_{i-1} + \IZ \lambda_i$, for
$ i \in \{ 1,\, \dots,\, g \} .$ 
We have that $M_0\subset M_1 \subset \ldots \subset M_g$ and we set
$$
	n_i := [M_i : M_{i-1}], \hspace{10pt} i \in \{ 1,\,\ldots ,\, g \},
$$
where $[M_i : M_{i-1}]$ denotes the index of the subgroup $M_{i-1}$ in $M_i.$

The importance of these exponents comes from the fact that, when $\field=\IC $
and $ f $ is complex analytic, Gau \cite{G} proves
that they determine the topological type of $(V,0).$

We can also define an equivalent data to the characteristic exponents (\cite{KM}, \cite{GP1}), as  follows
\begin{equation}
\label{eq:qo_def_lamda}
	{\gamma}_1 =  \lambda_1 
	\,\,\mbox{ and }\,\,
	{\gamma}_{i+1}- n_i {\gamma}_{i} = \lambda_{i+1} -  \lambda_{i}
	\hspace{10pt}
	\mbox{ for } i \in \{ 1, \,\ldots, \,g-1 \},
\end{equation}
%

%
%
%
%
%
%
%
%
%
%
%
%
%
%
%
%
%

\medskip

\section{Weighted characteristic polyhedra}

For our characterization we introduce polyhedra with respect to a linear map $ \weight $.
These polyhedra play a crucial role in the construction of our invariant and are a generalization of Hironaka's characteristic polyhedra,
 \cite{H} or \cite{CScompl}, section 1.

\medskip

Let $ f \in \field [[ \x ]] [ z ] $ a Weierstrass polynomial of degree $ n \in \IZ_+ $ in $ z $, i.e.,
\begin{equation}
\label{definition_of_n}
	f ( \x, z ) =  z^n + f_1 (\x) \, z^{ n- 1} + \ldots + f_n (\x)
\end{equation}
for some $ f_i (\x) \in \field [[\x]] $, $ i \in \{ 1, \ldots, n \} $.
Consider an expansion of $ f $ of the form
$$
	f = \sum_{\a, b} \coeffforf_{\a, b } \, \x^\a \, z^b,
$$
for certain $ \coeffforf_{\a, b} \in \field $, and $ \a \in \IZ^d_\gqz $, $ b \in \IZ_\gqz $.

\medskip

Further, let
$$
	\weight : \IZ^d_\gqz \to \IQ^c_\gqz \,,
$$
for some $ c \in \IZ_+ $, be the linear map defined by
$$
	\weight ( e_1 ) := \alpha_1 \, , \hspace{5pt}
	\weight ( e_2 ) := \alpha_2 \, , \hspace{5pt}
	\ldots \, , \hspace{5pt}
	\weight ( e_d ) := \alpha_d \, , \hspace{5pt}
$$
for certain non-zero $ \alpha_1, \ldots, \alpha_d \in \IQ^c_\gqz $, $ \alpha_i \neq (0 ,\ldots , 0 ) $, and, for $ i \in \{1, \ldots, d \} $, $ e_i $ denotes the $ i $-th canonical basis vector of $ \IZ^d_\gqz $ (i.e., $ e_i = ( \delta_{ij} )_{j \in \{ 1, \ldots, d \}} $ and $ \delta_{ij} $ is one if $ j = i $ and zero for all $ j $ with $ j \neq i $).

Clearly, this is the linear map {determined} by the $ c \times d $ matrix, also denoted by $ W $, with column vectors $ \alpha_1, \ldots, \alpha_d $,
$$ 
	\weight = ( \, \alpha_1 \, | \, \alpha_2 \, | \, \cdots \, | \, \alpha_d \,) \,.
$$
For $ \a = ( a_1, \ldots, a_d ) \in \IZ^d_\gqz $, we have
$
	\weight (\a ) = 
	\alpha_1 a_1 + \ldots + \alpha_d a_d .
$

\begin{Rk}
	One may consider $ \weight $ also as a map from the monomials of $ \field [[ \x ]] $ to $ \IQ^c_\gqz $.
	This means $ \weight $ assigns to $ x_i $ the ``weight'' $ \alpha_i \in \IQ^c_\gqz $, for all $ i \in \{ 1, \ldots, d \} $.
	In the following we sometimes write also $ \weight ( \x^\a ) $ when we mean $ \weight ( \a ) $.
	We refer also to the first section in \cite{T}.
\end{Rk}
 
\begin{Ex}
\label{Ex:nu}
\begin{enumerate}
	\item
		A first example is the map given by the identity, i.e., $ c = d $ and
		$
			\weight ( \a ) =  \a .
		$
		In the following, we denote this special example by $ \weight_0 $.

	\medskip
	
	\item
		Let $ d = 3 $, and $ c = 2 $.
		Then another example for such a map is the following $ \weight : \IZ^3_\gqz \to \IQ^2_\gqz $ (on $ \field [[ x_1, x_2, x_3 ]] $) which is given by
		$ \weight = 
		\left( \begin{array}{ccc}
			1	&	0	&	2	\\
			0	&	1	&	1
		\end{array}
		\right)
		$, i.e.,
		%
		%
		$		\weight ( x_1 ) = ( 1, 0 ), 
				\weight ( x_2 ) = ( 0, 1 ), 
				\weight ( x_3 ) = ( 2,1 ). $
 		%
 		%
\end{enumerate}
\end{Ex}

\begin{Def}
\label{Def:associated_Polyhedron}
	Let $ f = \sum\limits_{\a, b} \coeffforf_{\a, b } \, \x^\a \, z^b \in \field [[\x]] [z] $ be a polynomial of degree $ n $ and let $ \weight: \IZ^d_\gqz \to \IQ^c_\gqz  $ be a linear map as before.
	\begin{enumerate}
		\item	We define the {\it Newton polyhedron of $ ( f, \x, z ) $ \wrt $ \weight $} as the smallest convex subset of $ \IR^{c + 1}_\gqz $ containing all the points of the set
	%
	%
	$$
		\left\{ \,
			( \weight ( \a ), b ) + \IR^{c + 1 }_\gqz \; \mid\;  \coeffforf_{\a, b } \neq 0 \,
		\right\}
	$$
	%
	%
	and we use the notation $ \cnupoly{N, \weight}{f}{\x,z} $.
		\item	We define the {\it associated polyhedron for $ ( f, \x, z ) $ \wrt $ \weight $} as the smallest convex subset of $ \IR^c_\gqz $ 
		containing all the points of the set
	%
	%
	$$
		\left\{ \,
			\frac{ \weight ( \a ) }{ n - b } + \IR^c_\gqz \; \mid\;  \coeffforf_{\a, b } \neq 0 \, \wedge \,  b < n \,
		\right\}
	$$
	%
	%
	and we use the notation $ \nupoly{\weight}{f}{\x}{z} $.
	\end{enumerate}
\end{Def}

\begin{Rk}
\begin{enumerate}
	\item 
		Based on discussions of the authors a variant of this idea was already mentioned in \cite{BerndCharPoly}, Remark 5.9.
Besides that there is no reference known to the authors, where such a kind of polyhedron has been considered before. A connected notion has been considered 
in \cite{As}.
	
	\item
		The polyhedron $ \nupoly{\weight}{f}{\x}{z} $ has finitely many vertices. 
(This follows with the same arguments as in \cite{CParithm}, Proposition 2.1).

		Further,
		one sees easily that the associated polyhedron $ \nupoly{\weight}{f}{\x}{z} $ is the projection of the Newton polyhedron $ \cnupoly{N, \weight}{f}{\x,z} $ from the point
		$ ( \0, n ) \in  \IR^{c + 1 }_\gqz $
		onto $ \IR^c_\gqz $ corresponding to the cooridnates $ ( \x ) $,
		followed by a homothecy of factor $ \frac{1}{n} $.
	
	\item 
		{If we consider $ \weight $ as a linear map from $ \IR^d $ to $ \IR^c $} then we have $ \frac{ \weight ( \a ) }{ n - b } =   \weight \left ( \frac{ \a }{ n - b } \right) $.
		Hence 
		\begin{equation}
		\label{eq:W_poly=W(poly)}
				 \nupoly{\weight}{f}{\x}{z} = \weight \left( \nupoly{}{f}{\x}{z} \right) + \IR^c_{\geq 0 }, 
		\end{equation}
		where $ \nupoly{}{f}{\x}{z} = \nupoly{\weight_0}{f}{\x}{z} $ is the polyhedron associated to $ ( f, \x, z ) $ (see also Definition 1.2 in \cite{CScompl}).
\end{enumerate}
\end{Rk}


\begin{Ex}
	Let $ d = 3 $, $ c = 2 $, and $ \weight = 
		\left( \begin{array}{ccc}
			1	&	0	&	2	\\
			0	&	1	&	1
		\end{array}
		\right) $ 
	the linear map defined in Example \ref{Ex:nu}(2).
	Consider
	$$
		f = z^2 + 2z x_1 x_2^3 + x_1^2 x_2^6 + x_1^3 x_3 + x_2^2 x_3^3.
	$$
	Then we have
	$$
	\begin{array}{ccl}
				\weight ( x_1 x_2^3 ) = ( 1, 3 ),
				&
				\Rightarrow
				&
				v_1 := ( 1, 3)\, , 
				\\[3pt]
				\weight ( x_1^2 x_2^6 ) = ( 2, 6 )&
				\Rightarrow
				&
				v_2 := v_1 = ( 1, 3) \, , 
				\\[3pt]
				\weight ( x_1^3 x_3 ) = ( 5, 1 ),&
				\Rightarrow
				&
				v_3 := ( \frac{5}{2}, \frac{1}{2}) \, ,  
				\\[3pt]
				\weight ( x_2^2 x_3^3 ) = ( 6, 5 )&
				\Rightarrow
				&
				v_4 := ( 3, \frac{5}{2}) \, ,
			\end{array}			
	$$	
	where $ v_i $ denotes the corresponding point $ \frac{ \weight ( \a ) }{ n - b } $ in the polyhedron.	
	Thus the vertices of $ \nupoly{\weight}{f}{\x}{z} $ are $ v_1 = ( 1, 3) $ and $ v_3 = ( \frac{5}{2}, \frac{1}{2}) $, whereas $ v_4 = ( 3, \frac{5}{2}) $ lies in the interior.
	
	If we put
	$
		y := z + x_1 x_2^3
	$,
	then
	$
		f = y^2 + x_1^3 x_3 + x_2^2 x_3^3.
	$
	Clearly, the point $ v_1 $ does not appear in $ \nupoly{\weight}{f}{\x}{y} $ which has only $ v_3 $ as vertex. 
	Moreover, $ v_3 $ can not be eliminated by a further change of $ y $.
	This means the polyhedron $  \nupoly{\weight}{f}{\x}{y} $ is minimal \wrt the choice of $ y $,
	 in the sense that there is no choice $ \widetilde{y} $ for $ y $ such that $ \nupoly{\weight}{f}{\x}{\widetilde{y}} \subsetneq  \nupoly{\weight}{f}{\x}{y} $	
\end{Ex}

\medskip

\noindent
{\bf Choices for z:}
As we have seen, the polyhedron $  \nupoly{\weight}{f}{\x}{z} $ depends heavily on the choice of $ z $.
Thus we seek for $ \widetilde{z} $, where $ \widetilde{z} = z + h ( \x ) $, for some $ h(\x) \in \field [[\x]] $, such that 
$
 \nupoly{\weight}{f}{ \x }{ \widetilde{z} } \subset \IR^c_\gqz
$ 
becomes minimal \wrt inclusion.
For \qo singularities one can achieve this by considering so called $ P $-good coordinates 
(see for example \cite{ACLM_Newton}, Definition 3.1 and Lemma 4.6).

More generally, since $ \field $ has characteristic zero, we can choose $ z $ in such a way that it has maximal contact with $ f $. 
Thus we attain the desired change in $ z $ by performing the so called {\em Tschirnhaus transformation}: 
Given any $ z $ such that $ f $ is of degree $ n $ as in \eqref{definition_of_n}, we set 
\begin{equation}
\label{eq:Tschirnhaus}
	\widetilde{z} := z + \frac{1}{n} \cdot f_1 (\x ) . 
\end{equation}
This determines a hypersurface of maximal contact for $ f $ and we can replace $ z $ by $ \widetilde{z} $.
Then we get
$$
	f ( \x, z ) =  \widetilde{z}^n + g_2 (\x) \, \widetilde{z}^{ n- 2} + \ldots + g_n (\x)
$$
for some $ g_i (\x) \in \field [[\x]] $, $ i \in \{ 2, \ldots, n \} $, and $ g_1 (\x ) \equiv 0 $.
By Proposition 6.1 in \cite{BerndCharPoly} the associated polyhedron $ \nupoly{}{f}{\x}{\widetilde{z}} $ is minimal \wrt choices for $ \widetilde{z} $ and since $ \nupoly{\weight}{f}{\x}{\widetilde{z}} = \weight \left( \nupoly{}{f}{\x}{\widetilde{z}} \right)  + \IR^c_{\geq 0 } $, the same is true for $ \nupoly{\weight}{f}{\x}{\widetilde{z}} $.

\medskip

\noindent
{\bf Changes in (x):}
The shape of the polyhedron $ \nupoly{W}{f}{\x}{z} $ plays an central role in the construction of our invariant, Construction \ref{The_Construction}.
In particular, a good choice for the variables $ ( \x ) $ is essential if $ \weight = \weight_0 $ as the following example illustrates.

\begin{Ex}
	\label{Ex:changes_in_x_nec}
	Consider the polyhedron \wrt $ \weight_0 $ 
	for the polynomial $ f ( \x, z) = z^2 + x_1^2 ( x_1 + x_2) $.
	Then $ \nupoly{W}{f}{\x}{z} $ has two distinct vertices.
	On the other hand, if we choose the coordinates $ y_1 := x_1 $ and $ y_2 := x_1 + x_2 $, then we get that $ f ( \y, z ) = z^2 + y_1^2 y_2 $ and the corresponding polyhedron has exactly one vertex.
	For the importance of this difference, we refer to Construction \ref{The_Construction}, definition of $ \kappa_1 $. 
\end{Ex} 

Therefore one may ask if there exist choices for $ ( \x, z ) $ such that the polyhedron $  \nupoly{\weight_0}{f}{\x}{z} $ is minimal \wrt inclusion in $ \IR^c_\gqz $.
(Of course, one may ask this question in the more general case for arbitrary $ \weight$, but since this is becoming more complicated and since it is not needed here, we do not discuss this).

We want to make only changes in $ ( \x , z  ) $ which respect the projection of the singularity determined 
by $ f $ onto the affine space given by the variables $ ( \x ) $, i.e., on the level of rings which respect the inclusion
\begin{equation}
\label{eq:projection}
	\field[[ \x ]] \hookrightarrow \field [[\x ]][z] / \langle f \rangle .
\end{equation}
Thus we seek for $ ( \widetilde{ \x }, \widetilde{ z } ) $, where $ \widetilde{ x_i } = x_i + h_i ( \x ) $ and $ \widetilde z = z + h ( \x ) $, for some $ h_i ( x), h(x) \in \field [[\x]] $ ($ i \in \{ 1 , \ldots, d \} $), such that 
$
 \nupoly{\weight}{f}{\widetilde \x }{ \widetilde z } \subset \IR^c_\gqz
$ 
becomes minimal \wrt inclusion.
(Note that we use for $ ( \widetilde \x ) $ also the map $ \weight_0 $ induced by the identity; in particular, we do not require any compatibility with the original $ \weight_0 $ defined by $ ( \x ) $; 
for example, we allow a change of the form $ \widetilde x_2 := x_2 + x_1 $).

The coordinates $ (\widetilde \x, \widetilde z ) $ can be constructed in the following way:
We equip $ \IQ^c_\gqz $ with any total ordering, e.g., the one given by the lexicographical order of the entries.
First, we change $ z $ to $ \widetilde{z} $ as above such that $ \nupoly{\weight_0}{f}{\x}{\widetilde{z}} $ becomes minimal \wrt the choice of $ \widetilde{z} $.
{(Note that this is a change by an element in $ \field[[\x]] $, e.g.~ $ f = (z - x(1-x_1)^{-1} )^2 - x_2^3 $); here $(1-x_1)^{-1}=1+x_1+x_1^2+\cdots$.}
If $ \nupoly{\weight_0}{f}{\x}{\widetilde{z}} $ has only one vertex or is empty, then it is minimal \wrt the choice of $ ( \x ) $ and we are done.

Suppose the latter is not the case.
Consider the two smallest vertices and try to eliminate one of them by changes in $ ( \x ) $ respecting (\ref{eq:projection}).
If this is not possible, the smallest vertex is fixed and we compare the next two smallest vertices and so on.
{Note: In fact, for our constructions it is sufficient if we know that there are at least two vertices that are fixed. i.e., that cannot be eliminated by any changes in $ \x $.}

If we can eliminate one of the vertices by changing $ ( \x ) $, then we perform this change and start over again.
{Let us point out that we possibly have infinitely many changes of this kind (e.g., consider $ f = z^2 - x_1^3 (x_2 - x(1- x_3)^{-1}) $ and the changes in $ x_2 $ that we have to make). 
	But this is not a problem since we are allowed to change $ ( \x ) $ by elements in $ \field[[\x]] $ and when making infinitely many changes, by construction we build a series which is convergent in $ \field[[\x]]$.}
Note: Changes in $ ( \x ) $ do not affect the minimality \wrt the choice of $ \widetilde z $. More precisely, the coefficient of $  \widetilde  z^{n-1} $ always remains zero.

In the case of a \qo hypersurface singularity this process will end with a polyhedron that has exactly one vertex or is empty.
Thus in our context the shape of the obtained polyhedron will be unique.

\medskip

\begin{Rk}
	\begin{enumerate}
		\item
		Suppose $ z  $ is such that $ \nupoly{}{f}{\x}{z} := \nupoly{\weight_0}{f}{\x}{z} $ is minimal \wrt the choice for $ z $.
		Then the (unique!) polyhedron obtained coincides with Hironaka's characteristic polyhedron
		associated to $ ( f, \x ) $.
		(See for example, Definition 1.7 and Theorem 1.8 in \cite{CScompl}).
		By \eqref{eq:W_poly=W(poly)}, this implies that for any linear map $ \weight $ the polyhedron  $ \nupoly{\weight}{f}{\x}{z}  $ is also minimal \wrt the choice of $ z $.		
		\item
		One could make the definition of $ \nupoly{\weight}{f}{\x}{z} $ more general by allowing a whole system of variables $ ( \z ) = ( z_1, \ldots, z_c ) $ or by extending $ \weight $ to $ \field[[ \x ]][ \z ] $.
		In fact, one might even replace $ \field[[\x]][ \z] $ by a regular local ring $ R $ with regular system of parameters $ ( \x, \z ) $ and may consider instead of an element $ f $ an ideal $ J \subset R $.
		Similar ideas have been discussed in \cite{BerndCharPoly}, Remark 5.9.
		Since this is not important for our aim, we do not discuss these things in more detail.
	\end{enumerate}
\end{Rk}

\medskip

For the later use (Lemma \ref{Lem:Cycles_works_nu_qo}) we introduce in our special situation the notion of $ \nu $-\qo singularities \wrt a linear map $ \weight : \IZ^d_\gqz \to \IQ^c_\gqz $.
This generalizes the notion of $ \nu $-\qo polynomial which was introduced by Hironaka; more precisely, $ f $ is $ \nu $-\qo if it is $ \nu $-\qo \wrt $ \weight_0 $.

\begin{Def} 
	\label{Def:nu_qo_W}
	{Let $ f \in \field [[ \x ]] [ z ] $ be of degree $ n $ 
	and let $ \weight : \IZ^d_\gqz \to \IQ^c_\gqz $ be a linear map. 
}
{The polynomial $f$ is said to be {\em $ \nu $-quasi-ordinary \wrt $ W $} if there exists a choice for $ (\widetilde{\x},\widetilde{z}) $ respecting the inclusion 
		$\field[[\x]]\longrightarrow \field[[\x]][z]/\langle f \rangle $
		such that $ \nupoly{\weight}{f}{\widetilde \x}{\widetilde z} $ either is empty or has exactly one vertex.}
		
	{In this case, we also call the projection induced by $\field[[\x]]\longrightarrow \field[[\x]][z]/\langle f \rangle $ a {\em $ \nu $-quasi-ordinary projection \wrt $ W $}.}
\end{Def}

Suppose $ f $ is $ \nu $-\qo and $ \nupoly{\weight}{f}{\x}{z} $ has only one vertex $ v \in \IQ^c_\gqz $.
Since we have assumed $ f $ to be a Weierstrass polynomial, the Newton polyhedron $ \cnupoly{N, \weight}{f}{\x,z} $ must have a one-dimensional face starting from the point $ ( \0, n ) \in \IQ^{ c + 1 }_\gqz $ and projecting down to the point $ v $.
In particular, $ v $ can not be eliminated by changes in $ z $ and thus the same is true for the described face in the Newton polyhedron.
Clearly, this condition on the Newton polyhedron is equivalent to the one given in the definition.

For $ \weight = \weight_0 $ we obtain the usual definition of $ \nu $-\qo polynomials (see for example, \cite{ACLM_Newton}, Definition 1.11) for these special shaped $ f $.
The equality \eqref{eq:W_poly=W(poly)} shows that a $ \nu $-\qo polynomial is  $ \nu $-\qo polynomial \wrt any linear map $ \weight $.

On the other hand, 
there exist polynomials which are $ \nu $-\qo \wrt some $ \weight \neq \weight_0 $ but 
not $ \nu $-\qo in the usual sense.

\medskip

\begin{Def}
\label{Def:inti_v_W}
	Let $ f = \sum\limits_{\a, b} \coeffforf_{\a, b } \, \x^\a \, z^b  \in \field [[ \x ]] [ z ] $ 
	be of degree $ n $ as before.
	Let $ \weight : \IZ^d_\gqz \to \IQ^c_\gqz $, $ c \in \IZ_+ $, be the linear map determined by vectors $ \alpha_1, \ldots, \alpha_d \in \IQ^c_\gqz $.
	Let $ v \in \nupoly{\weight}{f}{\x}{z} $ be a vertex of the associated polyhedron. 
	
	We define the {\it initial form (or initial part) of $ f $ at $ v $ \wrt $ \weight $} by
	%
	%
	%
	$$
		F_{v,\weight} := in_{v,\weight} ( f ) := Z^n + \sum_{ (\a, b)\, :\, (\ast)} \coeffforf_{\a, b} \, X^\a \, Z^b \in \field [ \X, Z] ,
	$$
	%
	%
	where the sum ranges over those $ ( \a, b ) \in \IZ^{ d + 1 }_\gqz $ fulfilling
	$$
		\frac{ \weight ( \a ) }{ n - b } = v \, .
		\eqno{(\ast)}
	$$
\end{Def}

Since $ F_{v, \weight} $ lies in the graded ring \wrt $ \weight $ we use capital letter.
{Note that $ F_{v, \weight} \neq Z^n $ since $ v $ is a point appearing in the polyhedron.}

\begin{Ex}
\label{Ex:Initialnoteasy}
	Consider $ f = {z}^2 - x^{ 21 } - x^{18} {u}^3 \in \field [[ x, {u} ]][{z}] $.
	Let $ \weight : \IZ^2_\gqz \to \IQ_\gqz $ be the linear map defined by $ \weight ( x ) = 1 $ and $ \weight ( {u} ) = 1 $.
	Then $ \nupoly{\weight}{f}{\x{,u}}{{z}} $ has only the vertex $ v = \frac{21}{2} $ and
	$$
		F_{v,\weight} = {Z}^2 - X^{ 21 } - X^{18} {U}^3.
	$$
\end{Ex}

Of course, it was not used in the previous definition that $ f $ is irreducible. 
Moreover, we want to point out that $ f $ being irreducible does not imply $ F_{v,\weight} = ( Z^{n_1} + \coeffforf_1 \x^{\a_1} )^{e_1} $, for some $ \coeffforf_1 \in \field $, $ n_1, e_1 \in \IZ_+ $ and $ \a_1 \in \IZ^d_\gqz $:

In fact, the possible shape of $ F_{v,\weight} $  is connected with the question how many solution the linear system 
$
	W ( \a ) = v \in \IQ^{c}_\gqz 
$
has.
If there is a unique solution, say $ \a_1 \in \IZ^d_\gqz $, and if $ f $ is irreducible, then  $ F_{v,\weight} = (Z^{n_1} + \coeffforf_{1} \, \X^{\a_1})^{ e_1} $, for some $ \coeffforf_{1} \in \field $ and $ e_1, n_1 \in \IZ_+ $
(since $ \field $ is algebraically closed).
For example, this is the case for the usual $ \nu $-\qo polynomials,
Theorem 1.5 in \cite{ACLM_Factor},
{or Theorem 2.4 in \cite{GR_Bernd}}.
In the next section, we prove a similar result 
(Proposition \ref{Prop:ProductofBinomials})
for particular linear maps $ \weight $ appearing in our process.

%
%
%
%
%
%
%
%
%
%
%
%
%
%
%
%
%

\medskip

\section{The invariant and the main theorem}

We can now give the construction of our invariant.
The main result which we also state in this section is that the invariant detects whether a given irreducible hypersurface singularity is \qo or not.
But first, we need some simple techniques.

\medskip

Consider the ring 
$$ 
	R := \field[[\x]][u,z]/ \langle z - (u^{m} - q ( \x, u )) \rangle ,
$$ 
for some $ m \in \IZ_+ $ and $ q ( \x, {u} ) \in \field [[\x]]{[u]} $ with $ \deg_{u} (\, q (\x,{u})\,) < m $.
Let $ f $ be an element in $ \field[[ \x ]][{u}] $.
There is a unique (!) representative for the class of $ f $ in $ R $ contained in
\begin{equation}
\label{eq:restricted_powers}
	\field[[ \x ]][{u}]_{ < m }[{z}] := 
		\{\, 
			\sum_{  (\a,b,c) \in \IZ^{ d + 2}_\gqz } \, \coeffforfmodified_{\a,b,c} \, \x^\a \, {u^b} \, {z^c} 
			\;\mid\;
			 \coeffforfmodified_{\a,b,c} \in \field \, \wedge \, 
			 {0 \leq b < m }
		\,\} \, ,
\end{equation}
i.e., where the only powers of $ {u} $, which may appear, are $ {u, u^2, \ldots, u^{ m - 1}} $;
whenever $ {u^m} $ appears it is replaced by $ {z + q (\x , u)} $.

Thus we can identify $ \field[[ \x ]][{u}]_{< m}[{z}] $ with $ R = \field[[ \x ]][u,z] / \langle \, {z - ( u^m - q(\x,u) )} \, \rangle $, where we choose the {representative} of a class uniquely as an element in (\ref{eq:restricted_powers}).

\medskip

The important case for us is when $ q $ is a monomial, say $ q = \rho_\a \,\x^\a $, for some $ \a \in \IZ^d_\gqz $ and $ \rho_\a \in \field^\times $.
This case naturally arises in our characterization.
Namely, there appear irreducible polynomials $ f \in \field [[ \x]][u] $ which are $ \nu $-\qo \wrt some linear map $ \weight $ (on $ \field[[\x]] $) and whose initial form at $ v $ is
\begin{equation}
\label{eq:intial_good}
	F_{v,\weight} = (U^{m} - \coeffforf_{\a} \, \X^{\a})^{ e}  ,
\end{equation}
for some $ \coeffforf_{\a} \in \field^\times $ and $ e, m \in \IZ_+ $.
Here, we assume that $ \nupoly{W}{f}{\x}{u} $ is minimal for the choice of $ u $ and, further, $ v $ denotes its sole vertex. 

In order to detect more refined information on the variety defined by $ f $ we pass to a higher dimensional ambient space.
We do this via the embedding that is given by $ \field [[\x]][u,z] \to  \field [[\x]][u]$, where $ z $ is mapped to $ u^{m} - \coeffforf_{\a} \, \x^{\a} $.
Then we consider the unique representative $ \widetilde{f^+} \in \field[[ \x ]][{u}]_{ < m }[{z}] $ of the image of $ f $ in $ \field[[ \x ]][u,z] / \langle \, {z - ( u^m  - \coeffforf_{\a} \, \x^{\a} )} \, \rangle $.

We extend $ W $ to a linear map $ \weight_+$ on $ \field[[\x]][u] $ by setting $ \weight_+ ( u ) := v $.
If $ \widetilde{ f^+ } $ is $ \nu $-\qo \wrt $ \weight_+ $ then we would like to repeat the previous step.
The initial form of $ \widetilde{ f^+ } $ at the unique vertex is not necessarily of the shape \eqref{eq:intial_good} (as we can see in the following example), but we can choose in a canonical way a representative in $ R $ that is of the desired form.
 \begin{Ex}
 	\label{Ex:InitialNotGoodForm}
 	Let us have a look at 
 	$$ 
 	f = ( u^2 - x^3 )^4  - 2 x^5 u (u^2 - x^3)^2 + x^{13} \in \IC[x,u] .
 	$$
 	The polyhedron $ \nupoly{W_0}{f}{x}{u} $ has exactly one vertex $ v = \frac{3}{2} $ which can not be eliminated.
 	(Recall: $ W_0$ is given by the identity matrix).
 	The corresponding initial form is $ F_{v, \weight_0 } =  ( U^2 - X^3 )^4 $ and we set
 	$
 	z := u^2 - x^3.	
 	$
 	
 	Then $ \widetilde{ f^+ } =  z^4  - 2 x^5 u z^2 + x^{13} $ and it is $ \nu $-\qo \wrt $ W_+ $.
 	Note: $ W_+ ( u ) = \frac{3}{2} $.
 	Moreover, $ v_+ = \frac{13}{4} $ is the only vertex of the minimal polyhedron $ \nupoly{W_0}{\widetilde{ f^+ } }{x,u}{z}, $ and the initial form
 	$$
 	\widetilde{F^+}_{v_+, \weight_+ } = Z^4  - 2 X^5 U Z^2 + X^{13}
 	$$
 	is not a binomial as desired. 
 	But, after replacing $U$ by $X^{v_+},$ in the ring whose monomials are in the positive part of the lattice $ \IZ + \gamma \IZ $, $ \gamma := v = \frac{3}{2} $, the initial form (when the polyhedron have only one vertex)  can be written as a binomial to some power, namely $ (Z^2 - X^{\frac{13}{2}} )^2 .$ 
 	In order to obtain an honest polynomial we need to represent $ X^{\frac{13}{2}} $ by a monomial in  $\field [[x]][u]_{<2}[z],$ where the weight of $x$ is $1,$ the weight of $u$ is $3/2.$ 
 	By construction, there is a unique such a monomial which is $ x^5 u.$ Then the canonical form of $ f^{+}$ is given by 
 	$$
 	\begin{array}{rcl}
 	f^{+} & = & 
 	(z^2-x^5u)^2 + z^4  - 2 x^5 u z^2 + x^{13} - (z^2-x^5u)^2 =
 	\\[3pt]
 	& = & (z^2 - x^5 u )^2 - x^{10} z, 
 	\end{array}
 	$$
 	where $(z^2 - x^5 u )^2$ is the initial part that we keep as it is
 	and where we have put $ u^2 = z + x^3 $ in the remaining part.
  	
 	Note that the monomial $ x^{10} z $ corresponds to the point $ \frac{10}{3} = \frac{40}{12} > \frac{39}{12} = \frac{13}{4} $ in the polyhedron and hence lies in the interior.
 \end{Ex}
\smallskip

More generally, the arguments of the example can be applied to construct  an appropriate representative for the class of $ f $. 
Let us explain how this works.
Since we want to iterate this procedure, we introduce some notations.
Set $ v_1 : = v $.
For $ t \geq 0 $, 
suppose we have given 
$$ 
	v_1, \ldots, v_{t+1} \in \IQ^d_\gqz 
	\ \ \ \mbox{ and } \ \ \  
	n_1, \ldots, n_{t+1} \in \IZ_+ 
$$ 
such that, 
for all $ i \in \{ 0, \ldots, t  \} $,
\begin{itemize}
	\item $v_{i+1} >_{poly} n_{i}v_i$, if $ i > 0 $, and
	\item $v_{i+1} \not\in \IZ^d+\IZ v_1+\cdots+\IZ v_i$.
\end{itemize}

Let us begin with a useful lemma.
Let $ L $ be the lattice 
$$
	L := \IZ^d+\IZ v_1+\ldots+\IZ v_{t+1}.
$$
We denote by $L_{\geq 0}$ the semigroup of positive elements in $L.$ 
We define 
$$\field [\x^{L_{\geq 0}}] :=\{\sum_{finite}a_\alpha \x^{\alpha};a_\alpha \in \field,\alpha \in L_{\geq 0}\};$$ the set $\field [\x^{L_{\geq 0}}]$ has a natural ring structure.  We also consider the ring
$\field [[\x^{L_{\geq 0}}]]$ which is defined in a natural way as in the definition of $\field [\x^{L_{\geq 0}}] $, but where the sum may be infinite.

Let $ h \in \field [[\x^{L_\gqz }]][z] $ 
be a monic polynomial in $ z $ such that the monomials appearing in the coefficients of $z^l$ have exponents in $ L_\gqz $, the non-negative part of the lattice $ L $. 
Note that $ \IZ^d \subsetneq L \subsetneq \IQ^d $, 
i.e., we work with fractional exponents.

We naturally have a weight $ W_0: L \longrightarrow \IQ^d$ on these
monomials, which is simply defined by $W_0(\x^{\a})=\a $.
(Since it is the extension of $ \weight_0 $ on $ \IZ^d $, we also use the name $ \weight_0 $ here).
We write $ h = z^n + \sum h_i z^{n-i} $. 
As in Definition \ref{Def:associated_Polyhedron},
we can associate a polyhedron that we will denote again by $\nupoly{W_0}{h}{\x}{z} \subset \IR^d_\gqz $.

\begin{Lem}
\label{Lattice}	
	With the above notations, let $ h \in \field [[\x^{L_\gqz }]][z] $.
	If $ \nupoly{W_0}{h}{\x}{z} $ has a unique vertex $v_{t+2} $,
	then the initial part of $h$ at $v_{t+2}$
	(\wrt $ W_0 $)
	is a product of binomials, i.e., 
	it is of the form   
	$$
	H_{v_{t+2}} := in_{v_{t+2}, W_0} (h) = 
	\prod_{i=1}^{d_{t+2}} (Z ^{n_{t+2}} - \coeffforf_{i} \, \X^{\Alpha} )^{ e_{t+2,i}}
	$$
	for some $ \Alpha \in L $, 
	and $ n_{t+2} $, $ d_{t+2} $, $ e_{t+2,i} \in \IZ_+ $, 
	and $ \coeffforf_{i} \in \field $, pairwise different. 
\end{Lem}
 
\begin{proof}
	Let $(0,\ldots,0,n)$ and $(\a,i)$ be the extremities of the face (actually the segment) of the Newton polyhedron that correspond to $v_{t+2}$. 
	Let 
	$$
		(\b,c):=\frac{1}{\mu}(-\a,n-i) \in L 
	$$
	where $ \mu $ is the largest positive integer such that $\frac{1}{\mu}(-\a,n-i) \in L $. 
	Then 
	$$
		\X^{-\a}Z^{-i} H_{v_{t+2}}(\X, Z)
		=
		P(\X^\b Z^c),
	$$ 
	for a polynomial $P\in \field [t].$ Factorizing $P$ into a product of linear polynomials ($\field$ is algebraically closed) and 
	multiplying the result by $\X^{\a}Z^{i}$ provides the desired form. %
\end{proof}

Coming back to $ f \in \field [[\x]][z] $,
we will also repeatedly pass to a higher dimensional ambient space.
With the above notations, we set 
$ R_0 := \field[[\x]][u_0] $, for some $ u_0 := z + h_0(\x) $,
and, for $ t \geq 0 $, 
we define 
$$
	R_{t+1} := R_{t} [z_{t+1}]/ \langle z_{t+1} - (u_t^{n_{t+1}} - q_{t+1} ( \x, \u_{\leq t} ) ) \rangle,
$$
for some $ q_{t+1} \in \field[[\x]][\u_{\leq t}] $ with $ \deg_{u_{t}} (\, q_{t+1} \,) < n_{t+1} $,
	and $ u_{t+1} := z_{t+1} + h_{t+1} (\x, \u_{\leq t}) $ is some translation of $ z_{t+1} $.
Here and in the following, we frequently abbreviate:
\begin{equation}
	\label{eq:abbreviate_G_t+1}
	\left\{
	\begin{array}{ccl}
		\u_{\leq t}  & := & ( \, u_0, \,u_1, \, \ldots, \, u_{ t } \, ) 
		\\[5pt]
		\u_{< t}  & := & \u_{\leq t-1}
		\\[5pt]
		\x^{\a} \u_{< t}^{ \b } & := & 
		\x^{\a}\, u_0^{\b_{1} } \, u_1^{\b_{2} } \cdots u_{t - 1}^{\b_{t} } ,
	\end{array}
	\right.
\end{equation}
for $ \a \in \IZ^d_\gqz $ and $ \b \in \IZ^t_\gqz $.
Note that $ u_0 = u $, $ z_1 = z $, and $ R_1 = R $.

Similarly as we did for $ R_1 $, we can identify $ R_{t+1} $ with
$$
	\field [[\x]][u_0]_{< n_1}[u_1]_{< n_2} \cdots [u_{t}]_{< n_{t+1} } [ {z_{t+1}} ]
$$
which is defined analogously to \eqref{eq:restricted_powers}.
Furthermore, we equip $ \field[[\x]] $ with the linear map $ \weight_0 : \IZ^d_\gqz \to\IQ^d_\gqz $ and we extend it to a linear map on $ \field[[\x]][\u_{\leq t}] $ 
using $	v_1, \ldots, v_{t+1} $,
\begin{equation}
\label{eq:extension_of_weight}
\left\{
\begin{array}{ccll}
&&\weight_{t+1} : \IZ^{d + t {+1}}_\gqz \to \IQ^d_\gqz \\[3pt]
\weight_{t+1} ( x_i ) &:= &\weight_0 ( x_i ) \,, & \mbox{ for all } 1 \leq i \leq d \, ,
\\[3pt]
\weight_{t+1} ( u_{j }) &:= &v_{j+1} \in \IQ^d_\gqz \, , & \mbox{ for all } 0 \leq j \leq {t}  \, .
\end{array} 
\right.		
\end{equation}

We denote by $\widetilde{f^{(t+1)}} $ the unique representative of the image of $ f $ in $ R_{t+1} $ which is given by an element in 
$ \field [[\x]][u_0]_{< n_1}[u_1]_{< n_2} \cdots [u_{t}]_{< n_{t+1} } [ {z_{t+1}} ] $.

\begin{Prop}
		\label{Prop:ProductofBinomials}
		Let the notations be as before.
		Suppose that the associated polyhedron 
		$ \nupoly{ \weight_{t+1} }{ \widetilde{f^{(t+1)}} }{ \x, \u_{\leq t} }{ z_{t+1} } $
		is minimal \wrt the choice of $ z_{t+1} $ and that it has a unique vertex, say
		$$ 
		v_{t+2} 
		\in \nupoly{ \weight_{t+1} }{ \widetilde{f^{(t+1)}} }{ \x, \u_{\leq t} }{ z_{t+1} }.
		$$ 
		There exists a canonical representative $f^{(t+1)}$ of $\widetilde{f^{(t+1)}}$
		(i.e., of $ f $)		
		in $ R_{t+1} $
		such that the initial form of $f^{(t+1)}$ at the vertex 
		$ v_{t+2} $ \wrt $ \weight_{t+1} $ 
		is a product of powers of binomials, i.e.,
		\begin{equation}
		\label{POB}		
		F^{(t+1)}_{v_{t+2},\weight_{t+1}} =
		in_{v_{t+2},\weight_{t+1}} ( f^{(t+1)}) =
		 \prod_{i=1}^{d_{t+2}} (Z_{t+1} ^{n_{t+2}} - \coeffforf_{t+1,i} \, \X^{\a} \U^\b_{< t + 1 } )^{ e_{t+2,i}},
		\end{equation}
		where $ \a \in \IZ^d_\gqz $, 
		$ \b \in \IZ^{t+1}_\gqz $ with 
		$ \frac{\weight_{t+1}(\a, \b)}{n_{t+2}} = v_{t+2} $,
		and $ \gcd(n_{t+2},\a,\b) = 1 $, 
		and $ n_{t+2} $, $ d_{t+2} $, $ e_{t+2,i} \in \IZ_+ $, 
		and $ \coeffforf_{t+1,i} \in \field $, pairwise different. 
\end{Prop}

\begin{proof}
	Let $ h  \in \field [[\x^{L_\gqz }]][z_{t+1}] $ be the polynomial obtained
	from $\widetilde{f^{(t+1)}}$ by changing the variables 
	$ u_i $ to $ \x^{v_{i+1}} $. 
	Then $ h $ satisfies the hypothesis of Lemma \ref{Lattice} and we obtain a factorization of $ H_{v_{t+2} }  = in_{v_{t+2}, W_0 } (h) $, i.e., using the notations of the cited lemma,
	$$
			H_{v_{t+2}} = 
			\prod_{i=1}^{d_{t+2}} (Z_{t+1}^{n_{t+2}} - \coeffforf_{t+1,i} \, \X^{\Alpha} )^{ e_{t+2,i}}.
	$$	
	
	For simplicity, we set
	$ 
		T := \field [[\x]][u_0]_{< n_1}[u_1]_{< n_2} \cdots [u_{t}]_{< n_{t+1} }.
	$
	By the definition of $ T $,
	there exists for every element 
	$
		\Alpha\in L_\gqz \cap \big( \IZ^d_\gqz + \IZ_\gqz v_1 + \ldots + \IZ_\gqz v_{t+1}\big) 
	$
	in the non-negative part of $L $, 
	a unique monomial $ M_\Alpha := M_\Alpha (\x, \u_{\leq t} ) $ in $ T $
	having the weight $\Alpha $, 
	$ \weight_{t+1} (M_\Alpha) = \Alpha $.
	This allows us to associate with $  H_{v_{t+2}} $ the following unique product of binomials in $ T $,
	$$ 
		f^* := 	
		\prod_{i=1}^{d_{t+2}} (z_{t+1}^{n_{t+2}} - \coeffforf_{t+1,i} \, M_{\Alpha} )^{ e_{t+2,i}} 
		 = \prod_{i=1}^{d_{t+2}} (z_{t+1} ^{n_{t+2}} - \coeffforf_{t+1,i} \, \x^{\a} \u^\b_{< t + 1 } )^{ e_{t+2,i}}
		\in T, 
	$$ 
	where $ \a $, $ \b $, $ n_{t+2} $, $ d_{t+2} $, $ e_{t+2,i} $,
	and $ \coeffforf_{t+1,i} $ are as in the statement of the proposition.

	We define $f^{(t+1)} \in R_{t+1}$ as follows: 
	first, we consider
	$$
		f^* + \widetilde{f^{(t+1)}}- f^*
	$$
	and then we obtain $f^{(t+1)}$ by replacing 	$\widetilde{f^{(t+1)}}-f^* $ with its unique representative 
	in $ T .$
	By construction, $f^{(t+1)}$ is a representative of $ f $ in $ R_{t+1} $ and its initial form at the unique vertex $ v_{t+2} $
	\wrt $ W_{t+1} $ is as desired.
\end{proof}

\begin{Not_num}
	\label{Not:star}
	In the situation of the previous proposition, 
	we say that {\em condition $ ( \bigstar ) $ holds} if the following are true:
	\begin{enumerate}
		\item[$(\star 0)$] $d_{t+2} = 1$,
		\item[$(\star 1)$] $v_{t+2}>_{poly} n_{t+1}v_{t+1}$, and
		\item[$(\star 2)$] $v_{t+2}\not\in \IZ^d+\IZ v_1+\cdots+\IZ v_{t+1}$.
	\end{enumerate}
\end{Not_num}

Note that conditions $ (\star 1) $ and $ (\star 2) $ guarantee that the semigroup generated by the vectors $ v_{1}, \ldots, v_{t+2} $ and the canonical vectors of $ \IZ^d $, is isomorphic to the semigroup of a quasi-ordinary hypersurface (see \cite{GP2}).

\smallskip

Now, we can give the construction of the invariant 
$$ 
	\theinvariant = ( \, \theinvariant_1 \, ; \, \ldots \, ; \, \theinvariant_g \, ; \,  \theinvariant_{g+1} \, ) , 
	\hspace{15pt}
	 g \in \IZ_\gqz ,
$$ 
which we use to characterize quasi-ordinary singularities.

\medskip

\begin{Constr}
\label{The_Construction}
	Let $ f = \sum\limits_{\a, b} \coeffforf_{\a, b } \, \x^\a \, z^b  \in \field [[ \x ]] [ z ] $ 
	be an irreducible polynomial of degree $ n $. 
	First, we consider the polyhedron of $ f $ \wrt $ \weight = \weight_0 $ on $ \field [[ \x ]] $.
	(Recall: $ \weight_0 ( \a ) = \a $).
	Set 
	$$
		\nupoly{0}{f}{\x}{z} := \nupoly{\weight_0}{f}{\x}{z} 
	$$
	As described in the previous section we minimize $ \nupoly{0}{f}{\x}{z} $ \wrt the choice of $ ( \x, z ) $ respecting the inclusion $ \field[[ \x ]] \hookrightarrow \field [[\x ]][z] / \langle f \rangle $.
	Hence we assume that {$ ( \x, u_0 ) $, $ u_0 := z + h_0 (\x) $}, are such that $ \nupoly{0}{f}{\x}{{u_0}} $ is minimal.
	
	We define the first entry of $ \theinvariant = \theinvariant(f;\x;z) $ by
	$$
		\theinvariant_1 := 
			\left\{
					\begin{array}{ll}
						 v_1  	\, , & \mbox{ if $ \nupoly{0}{f}{\x}{{u_0}} $ has exactly one vertex $ v_1 $,}
						\\[3pt]
						 \GOOD 	\, ,	 &	\mbox{ if  $ \nupoly{0}{f}{\x}{{u_0}} = \varnothing $ is empty,}
						\\[3pt]
						 \BAD	\, ,	&	\mbox{ else.}
					\end{array}
			\right.
	$$
	If we are not in the first case then the construction ends and $ \theinvariant := (\, \theinvariant_1 \,) $.
	
	Suppose $ \nupoly{0}{f}{\x}{{u_0}} $ has exactly one vertex $ v_1 $.
	Then we can write the initial form of $ f $ at the vertex $ v_1 $ (Definition \ref{Def:inti_v_W}) as
	$$
		F_{v_1, \weight_0 } = ( {U_0^{ n_1 }} - \coeffforf_0 \X^{\a_1} )^{e_1 } ,
	$$
	for certain $ n_1, e_1 \in \IZ_+ $, $ n = n_1 e_1 $, $  \coeffforoverweight_{0} \in \field^\times $ and $ {\a_1} \in \IZ^d_\gqz $ the unique solution for $ \weight_0 (a_1) = n_1 \cdot v_1 $.
	{\em Note that we have $ n_1 > 1 $ (and thus $ e_1 < n $) because otherwise we can eliminate the vertex $ v_1 $}.

	We consider
	$$ 
		{
		R_1 := \field[[ \x ]][u_0, z_1] / \langle \, z_1 - ( u_0^{n_1} - \coeffforoverweight_{0} \x^{\a_1}) \ \rangle 
		}
	$$
	In there, we have
	$
		z_1  = u_0^{ n_1 } - \coeffforoverweight_{0} \x^{\a_1}.
	$
	Let $ \widetilde{f^{(1)}} \in \field[[ \x ]][u_0]_{< n_1 }[z_1] $ be the unique representative of the class of $ f $ in $ R_1 $, as in (\ref{eq:restricted_powers}).
	Since $ f $ is a polynomial of degree $ n $ in $ {u_0} $ we obtain that $ { \widetilde{f^{(1)}} } $ is a polynomial of degree $ e_1 = \frac{n}{n_1} {<n}$ in $ {z_1} $.
	As in \eqref{eq:extension_of_weight}, we extend $ \weight_0 $ from $ \field[[\x]] $ to a linear map 
	$ \weight_1 : \IZ^{ d + 1 }_\gqz \to \IQ^d_\gqz $ on 
	$ \field[[\x]][{u_0}] $ by assigning the value $ v_1 \in \IQ^d_\gqz $ to $ u_0 $.
	This finishes the first cycle of the construction and
	we call $ {\widetilde{f^{(1)}}} $ the transform of $ f $ under the first cycle.

	Suppose we have finished $ {t+1} $ cycles, $ { t\geq 0} $, and the construction did not end so far. 
	For notational convenience we put
	$$	
		f^{(0)} := f \in \field [[ \x ]] [u_0]=: R_0
	\hspace{20pt}
		\mbox{ and }
	\hspace{20pt}
		e_0 := n.
	$$
	Then the first $ t {+1} $ entries of $ \theinvariant $ are given by $ ( \, v_1; \, v_2; \, \ldots; \, v_{t{+1}} \, ) $ and
	the given data is 
	$$
	\widetilde{f^{(t+1)}} \in \field [[\x]][u_0]_{< n_1}[u_1]_{< n_2} \cdots [u_{t}]_{< n_{t+1} } [ z_{t+1} ],
	$$
	which is of degree $ e_{t+1} $ in $ {z_{t+1}} $, and $ e_{t+1} < e_{t} < \ldots < e_1 < e_0 $.
	{Note: $ \widetilde{f^{(t+1)}} $ is the unique representative of $ f $ in the ring
	$$
		R_{t+1} := R_{t}[z_{t+1}] / \langle \, z_{t+1} - ( u_{t}^{n_{t+1}} - \coeffforoverweight_{t} \x^{\a} \u_{<t}^{\b}) \, \rangle.
	$$
	(Recall the abbreviations \eqref{eq:abbreviate_G_t+1}).
	Clearly, the following equality holds in $ R_{t+1} $:
	\begin{equation}
	\label{eq:u_t+1_formula}
			z_{t + 1 }  = u_{t}^{n_{t+1}} - \coeffforf_{t} \, \x^{\a} \u^\b_{< t}.
	\end{equation}
	By the construction, we also have:
	\begin{equation}
	\label{eq:quasi_hom_condi}
		\frac{\weight_{t}( \, \a,\, \b \,) }{n_{t+1}} = v_{t+1} 
	\end{equation}
	}
	Further, $ W_0 $ got extended to a linear map 
	$ \weight_{t+1} : \IZ^{d + t +1}_\gqz \to \IQ^d_\gqz  $
	on $ \field[[\x]][\u_{\leq t}] $ by defining
	$ \weight_{t+1} ( u_{j} ) :=  v_{j+1} $, for $ 0 \leq j \leq t $,  as in \eqref{eq:extension_of_weight}.

	We consider the polyhedron associated to the given data
	and minimize it by changes in $ z_{t+1} $
	(e.g., via the Tschirnhaus transformation \eqref{eq:Tschirnhaus}),
	\begin{equation}
	\label{eq:minimizing}
		z_{t+1} \mapsto u_{t+1} := z_{t+1} + h_{t+1} ( \x, \u_{ \leq t } ),
	\end{equation}	 
	{and by changes in those $ x_i $ which did not appear in the monomials defining the re-embedding \eqref{eq:u_t+1_formula}
		(see Example \ref{Ex:Change_x_i_not_yet} below).}
	Suppose %
	$$
		\Delta_{t+1} := \nupoly{ \weight_{t+1} }{ \widetilde{f^{(t+1)}} }{ \x, \u_{\leq t} }{ u_{t+1} } \subset \IR^d_\gqz
	$$ 
	is minimal.
	We define
	$$
		\theinvariant_{t + {2}} := 
		\left\{
		\begin{array}{ll}
		v_{t+2} \, ,	 &	\mbox{ if  $ \Delta_{t+1} $ has exactly one vertex $v_{t+2} $ and $ (\bigstar) $ holds},
		\\[3pt]
		\GOOD 	\, ,	 &	\mbox{ if  $ \Delta_{t+1} = \varnothing $ is empty,}
		\\[3pt]
		\BAD 	\, ,	&	\mbox{ else.}
		\end{array}
		\right.
	$$ 
	In the first case, we denote by $ f^{(t+1)} $ the canonical representative of the class of $ f $ in $ R_{t+1} $, as in Proposition \ref{Prop:ProductofBinomials}.
	In particular, its initial form at  $ v_{t+2} $ (\wrt $ W_{t+1} $) is a product of powers of binomials \eqref{POB},
	and $ (\star 0) $ implies
	$$
		F^{(t+1)}_{v_{t+2},\weight_{t+1}} 
		= in_{v_{t+2}, \weight_{t+1} } ( f^{(t+1)} ) 
		= (U_{t+1}^{n_{t+2}} - \coeffforf_{t+1} \, \X^{\a} \U^\b_{< t+1} )^{ e_{t+2}}
		,
	$$
	where $ \frac{\weight_{t+1}(\a_{t+2}, \b_{t+2})}{n_{t+2}} = v_{t+2} $, $ \gcd(n_{t+2},\a_{t+2},\b_{t+2}) = 1 $, $ d_{t+2}, e_{t+2,i} \in \IZ_+ $.

	Note that $ n_{t+2} > 1 $ and thus $ e_{t+2} < e_{ t+1 } $.
	We define the ring
	$$
		R_{t+2} := R_{t+1}[z_{t+2}] / \langle \, z_{t+2} - ( u_{t+1}^{n_{t+2}} - \coeffforf_{t+1} \, \x^{\a} \u^\b_{< t+1} ) \, \rangle  .
	$$
	Finally, we denote by $ \widetilde{f^{(t+2)}} $ 
	the unique representative for the class of $ f^{ (t+1) } $ in $ R_{t+2} $ (which is also the class of $ f $) lying in
	$ 
	 	\field [[\x]][u_0]_{< n_1}[u_1]_{< n_2} \cdots [u_{ t + 1 } ]_{ < n_{ t + 2} }[z_{t+2}]
	$ 
	 and we extend $ \weight_{t+1} $ to $ \weight_{ t+ 2 } $ by putting $ \weight_{t+2} ( u_{ t+1 }) := v_{t +2} $.
	
	Since $ n = e_0 > e_1 > \ldots > e_t \geq 1 $, the construction ends after finitely many cycles, say after $ g \geq 1 $ cycles, and we obtain
	$$
		\theinvariant := \theinvariant(f;\x;z) := ( \, v_1; \, v_2; \, \ldots; \, v_g; \, \lastentry \,),
		\ \ \ \ \ \mbox{ with } \lastentry \in \{ \GOOD, \BAD \}.
	$$		
	\rightline{\it (End of Construction \ref{The_Construction}).}
\end{Constr}

\smallskip 

{Let us point out that, that there are two different minimizing processes for the polyhedron.
	First, the Tschirnhaus transformation in $ z $ and second, changes in the $ x_i $ variables that preserve the shape of the preceding polyhedra.
}
{After the first cycle, we perform only changes in those elements of $ ( \x ) $ that did not yet appear in the binomials \eqref{eq:u_t+1_formula}.
	The reason for this is that we do not want to spoil the polyhedra $ \Delta_j $, $ j \leq t $, that we considered in the preceding cycles. 
	It is essential that they remain with exactly one vertex during the procedure in order to obtain an interpretation of the given hypersurface as an overweight deformation of a binomial variety.
	In section 4, we deduce formulas for the roots of $ f $ from the latter which is one of the key ingredients for the proof of our characterization theorem.
	}
	
\begin{Ex}
	\label{Ex:Change_x_i_not_yet}
	{
			Consider $ f = (z^2 - x_1^3 )^2  - x_1^5 z (x_2^2 - x_3^2) \in \field[[x_1,x_2,x_3]][z] $.
			The projected polyhedron has precisely one vertex $ \gamma_1 := (\frac{3}{2},0,0) $ and we get 
			$$
				u_1 := z^2 - x_1^3
			$$
			and $\widetilde{ f^{(1)} } =  u_1^2 - x_1^5 z (x_2^2 - x_3^2) $. 
			At the first look, the weighted polyhedron has two vertices (using $ W(z) = \gamma_1 $).
			But, we are still allowed to make changes in $ x_2 $ and $ x_3 $, and for $ \widetilde{x}_2 := x_2 - x_3 $ and $ \widetilde{x}_3 := x_2 + x_3 $, we get
			$$
				f^{(1)} =  u_1^2 - x_1^5 \widetilde x_2 \widetilde x_3 z. 
			$$
			Hence, we can re-embed again and achieve $ \xi = \infty $
			which implies that $ f $ is a \qo polynomial, as we will prove later. 
	}
\end{Ex}


\begin{Rk}
\begin{enumerate}
\item
Roughly speaking, the characterization theorem below will state that if $ \theinvariant $ ends with $ \GOOD $, then the given $ f $ is a \qo polynomial.

On the other hand, if the last entry of $ \theinvariant $ is $ \BAD $ then $ f $ is not a \qo polynomial.
But still $ f $ might be \qo \wrt another choice of the projection, see Example \ref{Ex:different_projection}.

\smallskip

\item
In each cycle of the construction we embed the singularity into an ambient space with one dimension more.
From this we then obtain refined information on the hypersurface.
This is highly motivated by Teissier's approach to local uniformization for Abhyankar valuations, see \cite{T}.
His idea is to embed a given singularity into a suitable higher dimensional ambient space such that it becomes an overweight deformation of a toric variety.
Then a local resolution of the original singularity can be obtained from that of the associated toric variety.

\smallskip

\item
This construction is slightly connected to the invariant by Bierstone and Milman for constructive resolution of singularities in characteristic zero, see \cite{BMheavy} or \cite{BMlight}.
In \cite{BerndBM} the second author connected their invariant with certain polyhedra.
More precisely, 
$ \delta ( f ) := | v_1 | = v_{1,1} + \ldots + v_{1,d} $ is the resolution invariant for $ f $.
Therefore, $ v_1 \in \IQ^d_\gqz $ is a refined information of $ \delta (f ) \in \IQ_\gqz $
and $ v_{t + 1}  $ can be interpret as a generalization of $ \delta ( f^{ (t) } ) $ to a situation with weighted coordinates.

\smallskip

\item 
The assumption of $ f $ being irreducible implicitly appears when we state that the initial form at the vertex $ v_1 $ is a a power of a single binomial.
If $ f $ is not irreducible, this is not necessarily the case (e.g., see \cite{GR_Bernd}).
Nonetheless, it is possible to generalize the construction to arbitrary polynomials, but this requires more technical efforts that are not needed for the present result. 
This will be the subject of a future work. 

\smallskip 

\item
	In \cite{As}, Assi proves an irreducibility criterion for \qo polynomials.
	Different criteria of such type have been obtained by \cite{Evelia_Gwo_irred} or \cite{GV}.
	One may wonder if is possible to modify our construction in such a way that we can additionally detect if the given polynomial is irreducible.
	Since Construction \ref{The_Construction} is not fixed to \qo polynomials, we would obtain an irreducibility criterion for a larger class of polynomials (see also \cite{GR_Bernd}). 
	Hence this problem requires a more subtle study and will be the subject of a future work.
	The restriction of such a criterion to quasi-ordinary polynomials will give exactly Assi's criterion.
\end{enumerate}
\end{Rk}

	Let us briefly recall the notion of approximated roots, cf.~\cite{As} section 2, \cite{Abhy_semigp_mero_fct}, \cite{Patrick_approx_roots}.
	Given a polynomial $ f  = z^n + f_1 z^{n-1} + \ldots + f_n \in K[[\x]][z] $ and a monic polynomial $ g \in K[[\x]][z] $ of degree $ m $ such that $ n = dm $ for $ d \in \IZ_+ $.
	Then $ f $ can be uniquely written as
	$$
		f = g^d + a_1 g^{d-1} + \ldots + a_d \in K[[\x]][z]_{<m}[g]
		\ \ \ \ \ \ 
		(\mbox{cf.~\eqref{eq:restricted_powers}}).
	$$
	Recall that the Tschirnhaus transform of $ f $ (\wrt $ g $) is defined replacing $ g $ by $ \tau_f(g) := g + \frac{a_1}{d} $, 
	see \eqref{eq:Tschirnhaus}.
	(Note that $ d $ is invertible since $ \mbox{char}(K)= 0 $).
	The polynomial $ g $ is called a {\em $ d $-th approximate root of $ f $} if $ \deg_y (f-g^d) < n - m $.
	Note that the last condition is equivalent to $ \tau_f(g) = g $.

	The following draws a connection to the elements $ ( z, u_1, \ldots, u_g ) $ that we obtain from our construction.

\begin{Rk}[cf.~\cite{T}, Proof of Proposition 8.15, p.~546]
	\label{Rk:Approx_Roots}
	By construction, we have
	$$
		u_1 = z^{n_1} - \rho_1 \x^\a  + \sum \,
		\coeffforoverweightplus_{\Alpha, \beta_0} \, \x^{\Alpha} \, z^{\beta_0} ,
	$$
	where $ W_g(z^{n_1}) = W_g(\x^\a) <_\indexorder W_g (\x^{\Alpha} \, z^{\beta_0} ) $ if $ \coeffforoverweightplus_{\Alpha, \beta_0} \neq 0 $.
	Note that $ \beta_0 \leq n_1 - 1 $.
	
	If $ \coeffforoverweightplus_{\Alpha, n_1 - 1} = 0 $, for all $ \Alpha $, 
	then $ z $ is an $ n_1 $-th approximated root of $ u_1 $.
	For example, this is true if we minimize the polyhedra using the Tschirnhaus transform.
	
	Suppose $ \coeffforoverweightplus_{\Alpha, n_1 - 1} \neq 0 $ for some $ \Alpha $.
	We apply the Tschirnhaus transformation to eliminate all terms with $ \beta_0 = n_1 - 1 $,
	$$
		z \mapsto z_* := z + \frac{1}{n_1} 
		\sum \,
		\coeffforoverweightplus_{\Alpha, n_1 - 1} \, \x^{\Alpha} .
	$$
	We obtain $ u_1 = z_*^{n_1} - \rho_1 \x^\a  + \sum \,
	\widetilde{\coeffforoverweightplus}_{\Alpha, \beta_0} \, \x^{\Alpha} \, z_*^{\beta_0} $,
	for certain $ \widetilde{\coeffforoverweightplus}_{\Alpha, \beta_0} \in  \field $ that can only be non-zero if $ \beta_0 < n_1 $.
	This implies that $ z_* $ is a $ n_1 $-th approximated root of $ u_1 $.
	From the weight conditions, one can deduce that, after translation, we have
	$$ 
		\begin{array}{ccc}
			W_g(z_*^{n_1}) = W_g(\x^\a) <_\indexorder W_g (\x^{\Alpha} \, z_*^{\beta_0} ), \ \  \mbox{ if } \widetilde{\coeffforoverweightplus}_{\Alpha, \beta_0} \neq 0,
			\\[3pt]
			\Delta (f; \x; \z_*) = \Delta (f; \x; \z),
			\\[3pt]
			\Delta(f^{(1)}; \x, z_*; u_1 ) = \Delta(f^{(1)}; \x, z; u_1 ).
		\end{array}
	$$
	In particular, the two polyhedra remain minimal. 
	 
	 By repeating these arguments, we obtain a translation $ u_{1,*} $ of $ u_1 $ which is a $ n_2 $-th approximated root of $ u_2 $. Hence, $ z_* $ is also a $ (n_1 n_2) $-th approximated root of $ u_2 $.
	 Furthermore, all weight conditions remain true and the polyhedra do not change.
	 
	By continuing, we get elements $ u_{0,*}:= z_*, u_{1,*}, \ldots, u_{g-1,*} $ with the same properties as $ z, u_1, \ldots, u_{g-1} $ (weight conditions and minimal polyhedra)
	and additionally, we have that $ u_{t,*} $ is a $ (n_{t+1} \cdots n_g) $-th approximated root of $ f = u_g $, for $ 0 \leq t \leq g-1 $.
	
	If we use the Tschirnhaus transformation in the minimizing process during the construction, we have $ u_{t,*} = u_t $, for all $ 0 \leq t \leq g - 1 $ (with $u_0:=z$). 
\end{Rk}

	The previous observations indicate that the approximate roots of the polynomial $ f $ can play a determinant role in the construction of the invariant. 
	This may provide a variant adapted to a computational approach.
	Nonetheless this does not overcome the minimizing process of the polyhedron, since Tschirnhaus transformations (see \eqref{eq:Tschirnhaus}) also appear in the context of approximated roots.
	Finally, the presented construction provides a close connection to invariants for desingularization in the line of the second author's  work \cite{BerndBM} on the invariant for resolution of singularities in characteristic zero by Bierstone and Milman \cite{BMheavy}.

\medskip

Recall Definition \ref{Def:nu_qo_W}.
The following is an easy consequence of Construction \ref{The_Construction}: 

\begin{Lem}
\label{Lem:Cycles_works_nu_qo}
	Let $ f $ and $ ( \x, z) $ be as in Construction \ref{The_Construction}.
	We have
	$$ 
		\theinvariant(f;\x;z) = ( \, v_1; \, v_2; \, \ldots; \, v_g; \, \GOOD \,) 
	$$	
	if and only if the successive transforms $ {\widetilde{f^{(t)}}} \in \field [[\x]][u_0]_{< n_1}[u_1]_{< n_2} \cdots [u_{t-1}]_{< n_t } [ {z_t} ] $ of $ f $ are $ \nu $-\qo \wrt $ \weight_t $, for all $ t \in \{ 0, \ldots, g \} $.
	In particular, they are $ \nu $-\qo \wrt $ \weight_\ast := \weight_g $.
\end{Lem}

\begin{Obs}
\label{Obs:Darstellung_der_u_i}
	When the construction ends, we have 
	$$
		{\widetilde{f^{(g)}}} \in \field[[ \x ]] [u_0]_{<n_1} [u_1]_{ <n_2} \cdots[u_{g -1 }]_{< n_g} [ u_g]
	$$	
	and $ \nupoly{ \weight_g }{ {\widetilde{f^{(g)}}} }{ \x, u_0, \ldots, u_{g-1} }{ u_g } $ (with $ u_g $ such that the polyhedron is minimal) is either empty or has more than one vertex.
	If the polyhedron is empty, then $ {\widetilde{f^{(g)}}}= u_g $. 
	Otherwise $ f $ would not be irreducible. 
	
	Furthermore, we constructed an extension of $ \weight_0 : \IZ^d_\gqz \to \IQ^d_\gqz $ (on $ \field [[ \x ]] $) to
	$$
	\begin{array}{c}
		\weight_\ast := \weight_g : \IZ^d_\gqz \times \IZ^g_\gqz \to \IQ^d_\gqz
		\hspace{15pt} (\, \mbox{on } \field[[ \x ]] [ u_0, u_1, \ldots, u_{ g - 1 } ] \,),
		\\[5pt]
			\weight_\ast ( u_j) := \weight_s ( u_j ) = v_{j + 1},
			\hspace{15pt} \mbox{ for all } 0 \leq j \leq g - 1 ,
	\end{array}	
	$$
	where $ v_{j + 1} $ is the unique vertex of $ \nupoly{ \weight_j }{ {f^{(j)}} }{ \x, u_0, \ldots, u_{j-1} }{ u_j } $ 
	(and $ u_j $ is assumed to be chosen such that the polyhedron is minimal).
	The linear map $ W_\ast $ is determined by the  $ d \times ( d + g ) $ matrix with the $ d \times d $ identity matrix $ Id_d $ at first and then the matrix with columns given by the $ v_{j + 1} $, $ 0 \leq j \leq g - 1 $, i.e.,
	\begin{equation}
	\label{eq:matrix_W_ast}
		W_\ast = (\,Id_d \,| \, v_1 \,|\, v_2 \,|\, \cdots \,|\, v_g \,) .
	\end{equation}

	Recall the abbreviation \eqref{eq:abbreviate_G_t+1}.
	Then, using (\ref{eq:u_t+1_formula}) and taking into account that we do some minimizing process \eqref{eq:minimizing} for the polyhedron $ \nupoly{ \weight_{ t + 1 } }{ {\widetilde{f^{(t + 1)}}} }{ \x, u_{< t+1} }{ u_{t + 1 } } $, we obtain, for every $ t \geq 0 $,
	\begin{equation}
	\label{eq:u_t+1}
		u_{t + 1 }  = u_t^{ n_{t+1} } -
	 \coeffforoverweight_t \, \x^ {\a} \u_{< t}^{ \b }+ 
	 \sum_{(\ref{eq:quasi_hom_condi_starplus}) } 
	 \coeffforoverweightplus_{\Alpha, \Beta, \beta_+} \, \x^ {\Alpha} \, \u_{< t }^{ \Beta } \, u_t^{\beta_+} \, ,
	\end{equation}
	where $ \a = \a_{t + 1} \in \IZ^d_\gqz $ and $ \b = \b_{t + 1} \in \IZ^t_\gqz $ is the unique solution for 
	\begin{equation}
	\label{eq:quasi_hom_condi_star}
		\weight_\ast ( \, \a ,\, \b \,)  =  n_{t+1} \cdot v_{t + 1}
	\end{equation}
	(note that this is a reformulation of (\ref{eq:quasi_hom_condi}) using $ \weight_\ast ( \a , \b ) = \weight_t ( \a , \b ) $ since $ \weight_\ast $ is an extension of $ \weight_t $)
	and the second sum ranges over those $ \Alpha  \in \IZ^d_\gqz $ and $ ( \Beta, \beta_+ )  \in \IZ^{ t + 1 }_\gqz $ which fulfill 
	%
	$$
		\frac{ \weight_\ast ( \, \Alpha,\, \Beta \, ) }{ n_{ t + 1 } - \beta_+ }
		 >_\indexorder v_{t+1}
	$$
	%
	By using 
	$
		\weight_\ast ( \, \Alpha,\, \Beta, \, \beta_+ \,) =  \weight_\ast ( \, \Alpha,\, \Beta \,) +  \beta_+ \cdot v_{t+1} $
	we obtain the equivalent condition 
	\begin{equation}
		\label{eq:quasi_hom_condi_starplus}
		 \weight_\ast ( \, \Alpha,\, \Beta, \, \beta_+ \,)
		 >_\indexorder  n_{t+1} \cdot v_{t+1}.
	\end{equation}
\end{Obs}

\bigskip

To illustrate the construction of the invariant, let us do some examples before continuing with our results.

\begin{Ex}
\label{Ex:InitialnoteasyLATER}
	Consider the curve determined by
	$$
		f = \left[ \, ( z^7 - x^7 )^2 - x^{21} - x^{18} z^3 \, \right]^2 + x^{43} \in \field [[x]][z]
	$$
	The construction tells us to start with $ \weight_0 $.
	As one computes easily, the polyhedron $ \nupoly{\weight_0}{f}{x}{z} \subset \IR_\gqz $ is minimal, hence $ u_0 = z $.
	Moreover, it has the only vertex $ v_1 = 1 $ and
	$$
		F_{v_1, \weight_0 } = ( U_0^7 - X^7 )^2 .
	$$
	{The attentive reader surely recognizes that $ \gamma_1 := v_1 \in \IZ $.
	This implies that $ U_0^7 - X^7 $ is not irreducible and hence $ F_{v_1, \weight_0 } $ is not the power of a single irreducible binomial as required in the construction. 
	Thus one of the assumptions must fail -- namely, the given $ f $ is not irreducible, see for example \cite{GR_Bernd} Theorem 2.4}.

	{If we ignore the latter and introduce the $ R_1 $ as in the construction then we have} $ z_1 = u_0^7 - x^7 $,
	$$
		{\widetilde{f^{(1)}}} = \left[ \, z_1^2 - x^{21} - x^{18} u_0^3 \, \right]^2 + x^{43}
	$$
	and the extension of $ W_0 $ to $ W_1 : \IZ^2_\gqz \to \IQ_\gqz $ is given by $ W_1 ( x ) = 1 , W_1 ( u_0 ) = 1 $.
	The polyhedron $ \nupoly{\weight_1}{{\widetilde{f^{(1)}}}}{x, u_0}{z_1} \subset \IR_\gqz $ has the single vertex $ v_2 = \frac{21}{2} $ and 
	$$
		{\widetilde{F^{(1)}}}_{v_2, \weight_1 } = ( Z_1^2 - X^{21} - X^{18} U_0^3 )^2 .
	$$
	(This looks familiar, see Example \ref{Ex:Initialnoteasy}).
	{Since we have $ x^7 = u_0^7 - z_1 $ in $ R_1 $ we can not proceed as in Example \ref{Ex:InitialNotGoodForm} in order to get the desired shape of the initial form}. 
\end{Ex}

\begin{Ex}
\label{Ex:QO_Example}
	Let us give an example of an irreducible hypersurface. 
	Consider
	$$ 
		f = ( z^2 - x_1 x_2 x_3 )^3 - x_1^ 6 x_2^7 x_3^3 z.
	$$ 
	The polyhedron $ \nupoly{\weight_0}{f}{x}{z} \subset \IR^3_\gqz $ is minimal, ${u_0 = z}$, and has only one vertex $ v_1 = ( \frac{1}{2}, \frac{1}{2}, \frac{1}{2} ) $.
	We have
	$$
		F_{v_1, \weight_0 } = ( {U_0^2} - X_1 X_2 X_3 )^3 .
	$$
	Hence we {have $ z_1 =   u_0^2 - x_1 x_2 x_3 $ in $ R_1 $}.
	Then we extend $ \weight_0 $ to $ \weight_1 $ by defining $ \weight_1 ( {u_0} ) := ( \frac{1}{2}, \frac{1}{2}, \frac{1}{2} ) $.
	We get
	$$
		{\widetilde{f^{(1)}}} = {z_1^3} - x_1^6 x_2^7 x_3^3 {u_0}
	$$
	and $ \nupoly{\weight_1}{{\widetilde{f^{(1)}}}}{x, {u_0}}{{z_1}} \subset \IR^3_\gqz $ has the single vertex $ v_2 = ( \frac{5}{2},  \frac{17}{6}, \frac{3}{2}) $, {$ u_1 = z_1 $.
	Furthermore, $ f^{(1)} = \widetilde{f^{(1)}} $ and}
	$$
		F^{(1)}_{v_2, \weight_1 } =   U_1^3 - X_1^ 6 X_2^7 X_3^3 {U_0}
	$$
	Therefore {we obtain $ u_2 = z_2 = {f^{(1)}} $ in $ R_2 $} and the construction ends:
	$$
		\theinvariant ( f; x; z ) = \left(\, ( \frac{1}{2}, \frac{1}{2}, \frac{1}{2} ); \, ( \frac{5}{2},  \frac{17}{6}, \frac{3}{2}); \, \GOOD \, \right).
	$$
\end{Ex}


\begin{Ex}
	We encourage the studious reader to apply the construction for
	$$
		f = ( z^3 - x_1^3 x_2 )^2 + 2 ( z^3 - x_1^3 x_2 )  x_1^3 x_2 x_3^7 z + x_1^6 x_2^2 x_3^{14} z^2 - x_1^{37} x_2^{11} x_3^{73} z.
	$$ 
\end{Ex}

\medskip

The main result of this article is the following characterization theorem.


\begin{Thm}
\label{MainTheorem_TEXT}
	Let $ f \in \field [[ \x ]] [ z ] $ be an irreducible Weierstrass polynomial of degree $ n $. 
	Then $ f $ is a  \qo polynomial if and only if the last entry of
	$ 
		\theinvariant(f;\x;z)
	$
	is $ \GOOD $.
\end{Thm}

\smallskip

	If the last entry of $ \theinvariant(f;\x;z) $ is $ \BAD $, then the theorem implies that $ \field[[\x]] \hookrightarrow \field[[\x]][z]/ \langle f \rangle $ is not a \qo projection. 
	But still it is possible that there is another projection such that the given singularity is quasi-ordinary.
	
	\begin{Ex}
	\label{Ex:different_projection}	
	For $ f = x_1^2 + z^3 + z^2 x_2 $ we have 
	$ \theinvariant ( f;\x;z) = \BAD $.
	On the other hand, if we pick $ ( y_1, y_2, w ) = ( z, x_2, x_1 ) $ then $ f = w^2 + y_1^3 + y_1^2 y_2 = w^2 + y_1^2 ( y_1  + y_2 ) $ is quasi-ordinary, 
	$ \theinvariant ( f;\y;w) = \left( \, (1, \frac{1}{2} ) ; \GOOD \, \right) $.
	\end{Ex}

We begin by proving the easy direction of the theorem. The other direction is left to the next section.

\begin{Prop}\label{sdir} Let $ f \in \field [[ \x ]] [ z ] $ be an irreducible Weierstrass polynomial of degree $ n .$ If $ f $ is \qo then the last entry of $ \theinvariant(f;\x;z) $ is $ \GOOD.$
\end{Prop}

\begin{proof}
	The main ingredient is the description of the equation of a quasi-ordinary
	polynomial $ f$ in terms of its approximate roots by Gonzalez-Perez \cite{GP1}. 
	First note that
	since the projection induced by $ k[[x]]\longrightarrow {\field[[\x]]}/{(f)}$ is quasi-ordinary, there
	exists a change of variables of the type $ x_i\mapsto x_i +h_i(\x), z\mapsto z+h(\x) $, with $ h_i,h \in \field[[x]]$ such that
	the polynomial obtained from $f$ after this change of variables is quasi-ordinary. 
	We assume that the coordinates
	$(\x,z)$ already satisfy this property (after a possible change of variables).  
	Let $n_j,\gamma_j,j\in\{1,\ldots,g\}$ be the invariants
	associated with $f$ as in the first section. 
	For $ j \in \{ 0,\ldots,g-1 \} $, let $ q_j $ be the approximate root of $ f $ of degree  $ \frac{n_1\cdots n_g}{n_{j+1}\cdots n_g}$
	and set $ q_g : = f $. 
	It follows from \cite{GP1} Lemma 35 that,
	for $ j \in\{ 1,\ldots,g \} $, the approximate roots $q_j$
	satisfy
	$$
		c_j^*q_j
		=
		q_{j-1}^{n_j}
		-
		c_j \, \x^{\a_j} \, \q_{<j-1}^{\b_j}
		+
		\sum
		c_{\Alpha, \Beta} \, \x^{\Alpha} \, \q_{<j }^{ \Beta }
	$$
	(similar to \eqref{eq:abbreviate_G_t+1}, we abbreviate
	$ \x^{\a_j} \, \q_{<j-1}^{\b_j}
	= x_1^{\a_{j,1}} \cdots x_{d}^{\a_{j,d}}
	q_0^{\b_{j,1}} \cdots q_{j-2}^{\b_{j,j-1}} $,
	etc.),
	where
	$ c_j^*$, $c_j\in\field^\times$,
	and
	$ c_{\Alpha, \Beta} \in \field $,
	and
	$ \a_j $, $\Alpha \in \IZ^d_\gqz $,
	and
	$ (\, \b_j, 0 \,)$, $\Beta_j \in \IZ^{j}_\gqz $ with
	$ \b_{j,i} $, $ \Beta_{j,i} \leq n_i $ (for $ i \in \{ 1, \ldots, j\} $),
	and,
	for all $ (\Alpha, \Beta ) $ with $ c_{\Alpha, \Beta}  \neq 0 $,
	$$
	n_j \gamma_j
	=
	\a_j + \b_{j,1} \gamma_1 + \ldots + \b_{j,j-1} \gamma_{j-1}
	<_\indexorder
	\Alpha + \Beta_1 \gamma_1 + \ldots + \Beta_{j} \gamma_{j}
	.
	$$
	Since $ f = q_g $, applying Construction \ref{The_Construction} shows
	that
	$$
		u_j = q_j,
	\ \ 
	\mbox{ for all } j \in \{ 0,\ldots,g \}  
	.$$
	Indeed, the only thing that we need   to verify is that there is no need to make changes in the $\x$ variables in order to minimize the polyhedra at the different steps of the algorithm; but since every polyhedron at each step of
	the construction  have already only one vertex,  any changes in $\x$ would at best not change the polyhedra or make the polyhedra larger. 
	Hence we obtain $\theinvariant(f;\x;z)=(\gamma_1;\ldots;\gamma_g ;  \GOOD).$
\end{proof}

As a direct application of the proof of Proposition \ref{sdir} 
and Remark \ref{Rk:Approx_Roots}
we obtain a new algorithm to determine approximate roots (or semi-roots) of a \qo polynomial :

\begin{Cor}\label{AR}Let $ f \in \field [[ \x ]] [ z ] $ be an irreducible \qo polynomial of degree $ n $ and let
		$$\theinvariant(f;\x;z) = ( \, v_1;\, v_2; \, \ldots; \, v_g; \, \GOOD \,) .$$
	If we use the Tschirnhaus transformation to minimize the polyhedra in each step of Construction \ref{The_Construction}, then the approximate roots of $f$ are given by the constructed $u_i $, $ i \in \{ 0,\ldots,g \} $.
\end{Cor}

%
%
%
%
%
%
%
%
%
%
%
%
%
%
%
%
%

\medskip

\section{Computing roots by using overweight deformations}
Finally, we show the remaining part for the
proof of Theorem \ref{MainTheorem_TEXT} 
in

\begin{Prop}
\label{Prop:Remain_Main}
	Let $ f \in \field [[ \x ]] [ z ] $ be an irreducible polynomial of degree $ n $.  
	Suppose the last entry of
	$ 
		\theinvariant(f;\x;z)
	$
	is $ \GOOD $.
	Then $ f $ is \qo.
\end{Prop}

The strategy of the proof is as follows:
First, we analyse the data provided by the assumption that the last entry of our invariant is $ \GOOD $.
More precisely, we recall briefly Observation \ref{Obs:Darstellung_der_u_i} and explain how this is connected with an overweight deformation.
From this we deduce a parametrization of the singularity which then yields a statement about the roots of $ f $ (Theorem \ref{AJ}).
In the final step, we study the parametrization closely and compute the difference of two roots of $ f $ explicitly.
Since the discriminant is determined by the square of these differences this implies then that $ f $ is quasi-ordinary.

\begin{center}
	--------------------------------------------
\end{center}

As we have seen in Observation \ref{Obs:Darstellung_der_u_i}, if the last entry of $ \kappa(f;\x,z) $ is $ \infty $, we constructed a linear map $ \weight_\ast $ which can be identified with a weight map 
$$
	W := W_\ast : \field [[\x]][{u_0},u_1,\ldots,u_{g-1}] \longrightarrow \IQ^d.
$$
By (\ref{eq:u_t+1}), the singularity $ V( f) $ can be identified with the fiber $ \mfX_1 $ above $ T = 1 $ of the map $ \mfX \longrightarrow \mbox{Spec}~\field[T],$ for $ \mfX = V ( F_0, \ldots, F_{g-1} ) $ the variety given by 
$$ 
	F_t \in \field [[\x]][{u_0},u_1,\ldots,u_{g-1},T] \,\,, \,\,\, t \in \{ 0,\ldots, g-1 \} \, ,
$$
where (using the abbreviation (\ref{eq:abbreviate_G_t+1}) and {Proposition \ref{Prop:ProductofBinomials}}) %
\begin{equation}
\label{OD}
F_t = T u_{t + 1 }  - u_t^{ n_{t+1} } +	\coeffforoverweight_t \, \x^ {\a} \u_{< t}^{ \b } -
	 T \cdot \left ( 
	 \sum_{(\ref{eq:quasi_hom_condi_starplus}) } \,
	 \coeffforoverweightplus_{\Alpha, \Beta,\beta_+} \, \x^{\Alpha} \, \u_{< t }^{ \Beta } \, u_t^{\beta_+} 
	 \right)\, ,
\end{equation}
where $ \a=\a_{t+1} \in \IZ_\gqz^d, 
 \b=\b_{t+1} \in \IZ_\gqz^t,
 \Alpha=\Alpha_{t+2} \in \IZ_\gqz^d, 
 (\Beta, \beta_+)=\Beta_{t+2} \in \IZ_\gqz^{t+1} $, and 
$ \coeffforoverweight_t, 
 \coeffforoverweightplus_{\Alpha, \Beta, \beta_+} \in \field. 
$
By convention, $ u_g = 0 $ and further we have that 
(using Notation \ref{Not:poly_ordering})
\begin{equation}
\label{weight}  
\left \{
\begin{array}{rrrcl}
	W( u_{t}^{n_{t+1}} )
	& = &
	W( \x^ {\a} \u_{< t}^{ \b } )
	& <_\indexorder &
	W( \x^{\Alpha} \, \u_{< t }^{ \Beta } \, u_t^{\beta_+} ) \,,
	\\[5pt]
	&&
	W( u_{t}^{n_{t+1}} ) 
	& <_\indexorder & 
	W ( u_{t+1} ).
\end{array}
\right.
\end{equation}
The first line follows by  \eqref{eq:quasi_hom_condi_star} and \eqref{eq:quasi_hom_condi_starplus}.
The second can be deduced from the fact
that $ v_{t + 1} $ is the only vertex of $ \nupoly{\weight_t}{f^{(t)}}{\x, u_0, \ldots, u_{t-1}}{u_t} $ (Construction \ref{The_Construction}), see also $ ( \bigstar) $.
 
For $ t \in \{ 0,\ldots, g-1 \} $, we set
$$
	\gamma_{t+1} := \weight ( u_{t} ) = v_{t+1}.
$$ 
By \eqref{eq:matrix_W_ast},
$ \weight $ is determined by the matrix $ (\, Id_d \,|\, \gamma_1 \,|\, \ldots \,|\, \gamma_{g} \,) $.
Hence \eqref{weight} can be rewritten as 
%
%
$$
\left \{
\begin{array}{rrrcl}
	n_{t+1} \gamma_{t+1}
	& = &
	\a + \sum\limits_{i = 1}^t \b_i \gamma_i
	& <_\indexorder &
	\Alpha + \sum\limits_{i=1}^{t} \Beta_i \gamma_i + \beta_+ \gamma_{t+1} 
	\,,
	\\[5pt]
	&&
	n_{t+1} \gamma_{t+1} 
	& <_\indexorder & 
	\gamma_{t+2} \,,
\end{array}
\right.
$$
%
where we write $ \b = ( \b_1, \ldots, \b_t ) $
and
$ \Beta = ( \Beta_1, \ldots, \Beta_t ) $.
Note that actually $ \b = \b_{t + 1} $ and $ ( \Beta, \beta_+ ) = \Beta_{t + 2} $ but in order to avoid too complicated expressions we use these references only when they are really needed.
Otherwise we suppress them.

Let $ \mfX_0 $ be the special fiber which is defined by the above equations where
we replace $ T $ by $ 0 $, i.e., with the notations of \eqref{OD}
\begin{equation}
\label{OD_T_zero}
	u_t^{ n_{t+1} } - \coeffforoverweight_t \, \x^ {\a} \u_{< t}^{ \b } = 0 
	\,\,\, , \,\, 
t \in \{ 0,\ldots, g-1 \} \,.
\end{equation} 
\noindent
Note that $ \mfX_0 $ is a toric variety.
Consider the point on $ \mfX_0 $ given by 
$$ 
	( x_1, \ldots, x_d, {u_0}, u_1, \ldots, u_{g-1}) = (1,\ldots,1,c_0,c_1,\ldots,c_{g-1});
$$
By using this in \eqref{OD_T_zero}, 
we can determine the entries $ c_0, \ldots, c_{g-1} $.
Namely,
\begin{equation}
\label{coeff_c_t}
\left\{
\begin{array}{rclll}
	c_0  & = & \eta_1 \cdot \coeffforf_0^{ \frac{ 1 }{ n_1 } } \,,& \mbox{where } \eta_1^{n_1} = 1 \\[3pt]
	c_t  & = & \eta_t \cdot \left( \coeffforf_t \cdot c_{<t}^{\b} \right)^{ \frac{ 1 }{ n_t } } \,,& \mbox{where } \eta_t^{n_t} = 1\,, & \mbox{for } t \geq 1,
\end{array}
\right.
\end{equation}
and we abbreviate again $ c_{<t}^\b = c_0^{ \b_{t+1,1} } \cdots c_{t-1}^{ \b_{t+1,t} } $, for $ \b = \b_{t+1}  \in \IZ_\gqz^t $.

For each root of unity $ \eta_t $ there are $ n_t $ choices and we get in total $ n_1 \cdot n_2 \cdots n_g = n $ choices for $ ( c_0, \ldots, c_{g-1} )$.
Later we will see that these $ n $ choices determine the $ n $ different roots of the original \qo hypersurface $ f $

If $ c_0 \cdots c_{g-1} \neq 0 $ then $(1,\ldots,1,c_0,c_1,\ldots,c_{g-1})$ belongs to the orbit of the torus of the toric variety, thus, it cannot belong to the singular locus
of $\mfX_0.$ 
Let $ S = (s_1,\ldots,s_d) $ be a set of $ d $ independent variables. 
A parametrization of $\mfX_0$ is then given by 
\begin{equation}
\label{PS}
\left \{
\begin{array}{ccll}
	x_i	& = & s_i^n 					\,, & i \in \{ 1,\ldots , d \}, \\[3pt]
	{u_0}	& = & c_0 \,S^{n\gamma_1} 		\,, & \\[3pt]
	u_t	& = & c_t \, S^{n\gamma_{t+1}}	\,, & t \in \{ 1,\ldots , g-1 \} .
 \end{array}
\right.
\end{equation}

Recall that, for $ i \in \{ 1, \ldots, d \} $, $ e_i $ denotes the $ i $-th unit vector in $ \IR^d $.

\begin{Lem}
\label{PG}
The germ $ \mfX_1 = V ( f ) $ has a parametrization of the form
\begin{equation}
\label{eq:para_T_one}
\left \{
\begin{array}{ccclll}
	x_i 	& = &	s_i^n & + & \sum\limits_{\M >_\indexorder n \cdot e_i}  X_{i,\M} \, S^\M 	 \,, & i \in \{ 1,\ldots , d \} ,\\[7pt]
	{u_0}	& = &	c_0 \, S^{ n \gamma_1} & + & \sum\limits_{\M >_\indexorder n \cdot \gamma_1} Z_{\M} \, S^{\M} 		\,, & \\[7pt]
	u_t	& = & 	c_t \, S^{n\gamma_{t+1}} & + & \sum\limits_{\M >_\indexorder n \cdot \gamma_{t+1}} U_{t,\M} \, S^{\M}	\,, & 	t \in \{ 1,\ldots, g-1 \} \,.
 \end{array}
\right.
\end{equation}
\end{Lem}

\begin{proof}
We are searching for a solution of the equations \eqref{OD} (with $T=1)$ in the power series ring
$
	\field [[S]]=\field [[s_1,\ldots,s_d]]
$ 
which lifts the solution of $\mfX_0$ given in \eqref{PS}.
Let $ f_r $, $ r \in \{ 0,\ldots , g -1 \} $, be the series obtained from $ F_r $ in \eqref{OD} by replacing $ T $ by $ 1 $,
$$
	f_r  :=  u_{r + 1 }  - u_r^{ n_{r+1} } +	\coeffforoverweight_r \, \x^ {\a} \u_{< r}^{ \b } -
	\sum_{(\ref{eq:quasi_hom_condi_starplus}) } \,
	\coeffforoverweightplus_{\Alpha, \Beta,\beta_+} \, \x^{\Alpha} \, \u_{< r }^{ \Beta } \, u_r^{\beta_+}.
$$

Let $ \eta : K[[\x]][u_0, \ldots, u_{g-1}] \to K[[S]] $ be the morphism defined by \eqref{eq:para_T_one}.
Then
%
%
$$
 	\eta(f_r) = \sum_{\N} f_{r,\N} \, S^{\N},
$$
%
where $ f_{r,\N} $ ($ r \in \{ 0,\ldots , g-1 \}, \N \in \IZ_\gqz^d $) are polynomials in the variables 
$ X_{i, \M} $ ($ i \in \{ 1,\ldots, d \} $), $ Z_{\M} $ and $ U_{t, \M} $ ($ t \in \{ 0,\ldots , g-1 \} $) with $ \M \in \IZ_\gqz^d $.

The weight conditions and the fact that \eqref{PS} gives a parametrization of $ \mfX_0 $
imply
\begin{equation}
\label{eq:already_zero}
	f_{ r, \N } \equiv 0, \ \ \ 
	\mbox{ for } \N \not\geq_\indexorder n_{r+1} n \gamma_{r+1} \mbox{ or } \N = n_{r+1} n \gamma_{r+1}.
\end{equation}
(Recall Notation \ref{Not:poly_ordering} in order to keep in mind what $ \not\geq_\indexorder $ means).
To conclude, we need to determine  
$ X_{i, \M} $, $ Z_{\M} $ and $ U_{t, \M} $ ($ i \in \{ 1,\ldots, d \} $, $ t \in \{ 0,\ldots , g-1 \} $, $ M \in \IZ^d_\gqz $) in such a way that all $ f_{r,\N} $ are zero. 
Let us have a look at $ f_{0,\N } $ which is coming from 
	$$
	f_0 = u_{ 1 }  - {u_0^{ n_{1} }} + \coeffforoverweight_0 \, \x^ {\a} -
	\left ( 
	\sum \,
	\coeffforoverweightplus_{\Alpha,\beta_+} \, \x^{\Alpha} \, {u_0^{\beta_+}} 
	\right) ,
	$$
	where the sum ranges over those $ (\Alpha, \beta_+ ) $ such that 
	\begin{equation}
	\label{eq:hwt}
	\Alpha + \beta_+ \gamma_{1} >_\indexorder n_{1} \gamma_{1}  . 		
	\end{equation}
	Using \eqref{eq:para_T_one}, we obtain
	$$
	\begin{array}{rl}
	\eta(f_0) = &	
	\big( 
		c_1 \, S^{n\gamma_{2}} + \sum\limits_{\M >_\indexorder n \cdot \gamma_{2}} U_{1,\M} \, S^{\M}
	\big)  
	- 
	\big( 
		c_0 \, S^{ n \gamma_1} + \sum\limits_{\M >_\indexorder n \cdot \gamma_1} Z_{\M} \, S^{\M} 
	\big)^{ n_{1} } 
	+			
	\\[5pt]
	&	
	+ 
	\big( 
		\coeffforoverweight_0 S^{n \a} + \sum\limits_{\M >_\indexorder  n\a } w_{0,\M} \, S^\M 
	\big) -
	\\[5pt]
	&
	- 
	\sum \,
	\big(
		\coeffforoverweightplus_{\Alpha,\beta_+} \, 
		S^{n \Alpha} + 
		\sum\limits_{\M >_\indexorder  n\Alpha } w_{\Alpha, \beta_+,\M} \, S^\M 
	\big) 
	\, 
	\big(
		c_0 \, S^{ n \gamma_1} + \sum\limits_{\M >_\indexorder n \cdot \gamma_1} Z_{\M} \, S^{\M} 
	\big)^{\beta_+} 
	\\[8pt]
	= & 
	\sum\limits_{\N} f_{0,\N} \, S^{\N},
	\end{array}
	$$
	where $ w_{0,\M} := w_{0,\M}(X_{i,\ast}) $ 
	(resp.~$ w_{\Alpha, \beta_+,\M} := w_{\Alpha, \beta_+,\M}(X_{i,\ast}) $) 
	are terms only depending on $ \coeffforoverweight_0 $ 
	(resp.~ $ \mu_{\Alpha, \beta_+}$) 
	and $ X_{i,\M} $. 
	Recall that $ n\a = n_1 n\gamma_1 $ (see \eqref{weight}).

	By \eqref{eq:already_zero}, we only need to consider those $ \N \in \IZ^d_\gqz  $ for which we have that $ \N = n_{1} n \gamma_{1} + \PL >_\indexorder n_{1} n \gamma_{1} $.
	We set $ \M (\N) := n \gamma_{1} + \PL $ and claim
	\begin{equation}
	\label{eq:claim_f_0_N}
	f_{0,\N} = 
	\left\{
	\begin{array}{rl}
	- n_1 c_0^{ n_1 - 1 } Z_{\M (\N)} + h_{<\M (\N)}	\,,					
	&	\mbox{ if } \N \not\geq_\indexorder n \gamma_2 \,, \\[3pt]
	c_1 - n_1 c_0^{ n_1 - 1 } Z_{\M (\N)} + h_{<\M (\N)}	 \,, 
	&	\mbox{ if } \N = n \gamma_2 \,, \\[3pt]
	U_{1,\N} - n_1 c_0^{ n_1 - 1 } Z_{\M (\N)} + h_{<\M (\N)}	\,,	
	&	\mbox{ if } \N >_\indexorder n \gamma_2 \,,
	\end{array}
	\right.			
	\end{equation}
	where $ h_{<\M (\N)} $ is a polynomial only in variables $ Z_\M $, for which $ \M <_\indexorder \M (\N ) $, 
	and $ X_{i,\M'} $, for any $ \M' $.
	The only part that one might has to think about is the property on $ h_{<\M (\N)} $.
	This can be seen with the following arguments:
	\begin{itemize}
		\item	Consider the term $ \binom{n_1}{j} \, c_0^{n_1 - j } \, S^{ n \gamma_1 ( n_1 -j ) } \left( \sum\limits_{\M >_\indexorder n \gamma_1 } Z_\M S^\M \right)^j $, for some $ j \in \{ 2, \ldots n_1 \} $.
		Then for each term with non-zero coefficient the exponent of $ S $ is of the form
		$$
		n \gamma_1 ( n_1 -j ) + \M_1 + \ldots + \M_j
		= 
		n \gamma_1 n_1 + \PL_1 + \ldots + \PL_j
		\,
		, 
		$$
		where $ \M_k = n \gamma_1  + \PL_k $, $ \PL_k >_\indexorder  \0 $, for all $ k \in \{ 1, \ldots, j \} $.
		This exponent coincides with $ \N = n_{1} n \gamma_{1} + \PL $ if and only if 
		$$
		\PL = \PL_1 + \ldots + \PL_j.
		$$
		But since all $ \PL_k $ are non-zero this implies $ \PL  >_\indexorder \PL_k $, for all $ k $, which on the other hand means
		$$
		\M_k =  n \gamma_1  + \PL_k  <_\indexorder n \gamma_1  + \PL = \M (\N ),
		$$ 
		for all $ k \in \{ 1, \ldots, j \} $.
		\item	By using \eqref{eq:hwt} and the previous arguments, one can show that we also have $ \widetilde\M <_\indexorder \M (\N) $ for those $ Z_{\widetilde\M} $ appearing in $ f_{0,\N } $ and coming from $ \sum \,
		\coeffforoverweightplus_{\Alpha,\beta_+} \, \x^{\Alpha} \, {u_0^{\beta_+}} $.
	\end{itemize}

	Set $ C := (n_1 c_0^{ n_1 - 1 })^{-1} $.
	Since we want to achieve $ f_{0,\N} = 0 $ for all $ \N $, we obtain from \eqref{eq:claim_f_0_N} that
	$$
		Z_{\M (\N)} = 
		\left\{
		\begin{array}{ll}
		C \cdot h_{<\M (\N)}	\,,					
		&	\mbox{ if } \N \not\geq_\indexorder n \gamma_2 \,, \\[3pt]
		C \cdot(c_1 + h_{<\M (\N)} )	 \,, 
		&	\mbox{ if } \N = n \gamma_2 \,, \\[3pt]
		C \cdot ( U_{1,\N} + h_{<\M (\N)})	\,,	
		&	\mbox{ if } \N >_\indexorder n \gamma_2 \,,
		\end{array}
		\right.	
	$$
		Therefore we understand $ Z_{\M (\N)} $ quite well in the first two cases, whereas we have to determine $ U_{1, \N } $ for the last one.
		Before we come to this, let us mention that $ \N = n \gamma_2 $ is equivalent to 
		%
		$$
		\M (\N) = n \gamma_1 + ( n \gamma_2 - n_1 n \gamma_1 ) 
		= n ( \gamma_2 - n_1 \gamma_1 + \lambda_1 ) = n \lambda_2,
		$$
		%
		where we use $ \lambda_1 := \gamma_1 $ and $ \lambda_2 := \gamma_2 - n_1 \gamma_1 + \lambda_1 $.
		(Compare the last definitions with \eqref{eq:qo_def_lamda}).

	In general, for $ r \geq 1 $, 
	{by \eqref{eq:already_zero}, we only have to consider 
		$$ 
			\N_r = n_{r+1} n \gamma_{r+1} + \PL >_\indexorder n_{r+1} n \gamma_{r+1} .
		$$   
	Further, the monomial $ u_r^{n_{r+1}} $ in $ f_r $ provides in $ \eta (f_r) $ those variables $ U_{r,\M_r (\N_r)} $ that will play the analogous role as $ Z_{\M (\N)} $ in \eqref{eq:claim_f_0_N}.
	More precisely, $ \M_r (\N_r) = n\gamma_{r+1} + \PL $ 
	and by eliminating $ \PL $, we get
	 	$$
	 	\M_r (\N_r) = n \gamma_{r+1} + \N_r - n_{r+1} n \gamma_{r+1}.
	 	$$
	One can show with the same arguments as above that}, for $ r \geq 1 $,
	\begin{equation}
	\label{eq:f_r,N}
	f_{r,\N_r} = 
	\left\{
	\begin{array}{rl}
	- n_{r+1} c_r^{ n_{r+1} - 1 } U_{r, \M_r (\N_r)} + h_{<\M_r (\N_r)}	,					
	&	
	\mbox{ if } \N_r \not\geq_\indexorder n \gamma_{r+2} \\[3pt]
	c_{r+1} - n_{r+1} c_r^{ n_{r+1} - 1 } U_{r, \M_r (\N_r)} + h_{<\M_r (\N_r)}	 \,, 
	&	
	\mbox{ if } \N_r = n \gamma_{r+2}, 
	\\[3pt]
	U_{r+1,\N_r} - n_{r+1} c_r^{ n_{r+1} - 1 } U_{r, \M_r (\N_r)} + h_{<\M_r (\N_r)}	\,,	
	&	
	\mbox{ if } \N_r >_\indexorder n \gamma_{r+2},
	\end{array}
	\right.			
	\end{equation}
	where
	$ h_{<\M_r (\N_r)} $ is a polynomial in variables $ X_{i,\M' } $ and those $ Z_\M $ and  $ U_{i,\M} $ that are already known.
	As before, we obtain a formula for the coefficients $ U_{r,\M_r(\N_r)} $ of $ u_r $.

Hence we can find a solution 
for the equations $ f_{r,\N_r} = 0 $ by giving values inductively on $ r $ to $ U_{ r , \M_r ( \N_r )} $, $ r \in \{ 0, \ldots , g-1 \} $ {starting from $ r = g-1 $}, where $ u_0 $ is $ z $.
This determines all the required coefficients $ Z_{\M} $ for $ \M >_\indexorder n \cdot \gamma_1 $ and $ U_{t, \M} $ for $ \M >_\indexorder n \cdot \gamma_{t+1} $ and $ t \in \{0, \ldots, g-1\} $, whereas the coefficients $ X_{i,\M} $ are free for $ \M >_\indexorder n \cdot e_i $ and $ i \in \{ 1, \ldots, d \} $. 
In particular, we can choose $X_{i,\M} = 0 $ for all $ i $, $ \M $, i.e., $ x_i = s_i^n $.
\end{proof}

\begin{Rk}
	The linearity argument in the previous proof is in the same spirit as the proof of the fact that the space of wedges on a smooth variety $ X $ is a locally trivial bundle over $ X $, see \cite{Y}.
\end{Rk}	

In fact, we can say more about the structure of $ u_0 $ in \eqref{eq:para_T_one}.

\begin{Lem}
	\label{Lem:struc_u_0}
	We have that
		%
		$$
		{u_0}	= 	c_0 \, S^{ n \lambda_1} 
		+ 
		\sum(c_0)
		+
		\sum_{r=1}^{g-1}
		\left( \,
		d_r ( c_0, \ldots, c_r ) \, S^{ n \lambda_{r+1}}
		+ 
		\sum(c_0, \ldots, c_r)
		\, \right)
		,	
		$$
		%
		where 
		\begin{enumerate}
			\item	$ d_r ( c_0, \ldots, c_r )  = d_{r,1}(c_0, \ldots, c_{r-1} ) \cdot c_r + d_{r,0}(c_0, \ldots, c_{r-1} ) $, for certain functions $ d_{r,1}, d_{r,0} $ in $ ( c_0, \ldots, c_{r-1} ) $, and
			\medskip
			\item	$ \sum(c_0) := \sum Z_{\M}(c_0) \, S^{\M} $ (resp.~$ \sum(c_0, \ldots, c_r) := (\sum Z_{\M}(c_0, \ldots, c_r) \, S^{\M})$) is an abbreviation for the intermediate terms whose coefficients do only depend on $ c_0 $ (resp.~$ (c_0, \ldots, c_r) $).
			By this we mean those $ \M $ with $ \M >_\indexorder n \lambda_1 $, but $ \M \not\geq n \lambda_2 $ (resp.~with $ \M >_\indexorder n \lambda_{ r+ 1} $, but $ \M \not\geq n \lambda_{ r + 2} $). 
		\end{enumerate} 
\end{Lem}

\begin{proof}
	We have that 
	$$ 
		\N_r >_\indexorder n_{r+1} n \gamma_{r+1} 
		\ \ \mbox{ and } \ \ 
		\M_r (\N_r) = n \gamma_{r+1} + \N_r - n_{r+1} n \gamma_{r+1}.
	$$
	(Cf.~\eqref{eq:para_T_one}).
	If we set $ \N_{r-1} := \M_r (\N_r) $, then $ U_{r,\N_{r-1}} $ appears with a unit as coefficient in the formula for $ U_{r-1, \M_{r-1}(\N_{r-1})} $,
	{more precisely, considering $ 
		f_{r-1,\N_{r-1}} = 0 $, \eqref{eq:f_r,N} (for $ r-1 $ instead of $ r $) provides that $ U_{r-1, \M_{r-1} (\N_{r-1})} $ is equal to
		$$
		\left\{
		\begin{array}{rl}
		( n_{r} \,c_{r-1}^{ n_{r} - 1 })^{-1} \cdot h_{<\M_{r-1} (\N_{r-1})}	,					
		&	
		\mbox{ if } \N_{r-1} \not\geq_\indexorder n \gamma_{r+1} \\[3pt]
		(n_{r} c_{r-1}^{ n_{r} - 1 })^{-1} \cdot \big(c_{r} + h_{<\M_{r-1} (\N_{r-1})} \big) 	 \,, 
		&	
		\mbox{ if } \N_{r-1} = n \gamma_{r+1}, 
		\\[3pt]
		(n_{r} c_{r-1}^{ n_{r} - 1 })^{-1} \cdot \big( U_{r,\N_{r-1}}  + h_{<\M_{r-1} (\N_{r-1})} \big)	\,,	
		&	
		\mbox{ if } \N_{r-1} >_\indexorder n \gamma_{r+1},
		\end{array}
		\right.			
		$$}
	
	\noindent 
	{Moreover, we have} 
	$$ 
	\begin{array}{l}
		\M_{r-1}(\N_{r-1}) 
		= 
		n \gamma_r + \N_{r-1} - n_r n \gamma_r 
		= n \gamma_r + n \gamma_{r+1} + \N_r - n_{r+1} n \gamma_{r+1} - n_r n \gamma_r =
		\\[5pt]
		\textcolor{white}{\M(\N_{r-1})}= 
		\N_r + n \big[ \gamma_r(1 - n_r) + \gamma_{r+1} (1 - n_{r+1})\big] . 
	\end{array}
	$$
	By iterating, we see that 
	$$
		\M_0(\N_0) = \N_r + n \big[ \sum_{i=1}^{r+1} \gamma_i (1 - n_i) \big].
	$$
	A simple induction (using \eqref{eq:qo_def_lamda}) shows that $ \gamma_{{r+1}} + \big[ \sum\limits_{i=1}^{{r}} \gamma_i (1 - n_i) \big] = \lambda_{{r+1}} $, and hence, 
	$ \N_{{r-1}} = n \gamma_{{r+1}} $ is equivalent to
	%
	$$
	\M_0 (\N_0) =  n \lambda_{{r+1}}.
	$$

	If we translate this to the coefficients of $ u_0 $, we see that 
	$$ 
		Z_{\M_0(\N_0)} =  a_0 U_{r,N_{r-1}} + a_< ,
	$$ 
	for some constant $ a_0 \neq 0 $ and a function $ a_< $ in coefficients $ U_{i,\M } $ that are defined before.
	Combining this with the preceding consideration, we obtain the assertion.
\end{proof}

The following result is a crucial consequence of Lemma \ref{PG} since it is the analogon of the Abhyankar-Jung Theorem in our situation.

\begin{Thm}\label{AJ}
 Let $ f\in \field[[x]][z] $ be an irreducible Weierstrass polynomial of degree $ n $.
 Assume that the last entry of $ \theinvariant (f; \x; z ) $ is $ \GOOD $.
 Then the roots of $ f $ as polynomial in $ z $ are contained in 
 $ \field [[x_1^{\frac{1}{n}}, \ldots, x_d^{\frac{1}{n}}]] $.
\end{Thm}

\begin{proof}
	Note that we can write $ x_i = s_i^n \cdot \epsilon_i $ in the parametrization found in Lemma \ref{PG}, where $ \epsilon_i $ is a unit in $ \field [[S]] $, for all $ i \in \{ 1,\ldots , d \} $.
	Since we are in characteristic zero, we can extract a $n-$th root of $ \epsilon_i $ and after a change of
variables we can write $ x_i = s_i^n $, $ i \in \{ 1,\ldots,d \}$.
	Hence $ s_i = x_i^{\frac{1}{n}} $, $ i \in \{ 1,\ldots,d \} $.
	After replacing all $ s_i $ by $ x_i^{\frac{1}{n}} $ in the expansion of {$ z = u_0(S) + h_0(\x(S)) $} in Lemma \ref{PG}, we obtain that $ f $ has a root in  $ \field [[x_1^{\frac{1}{n}},\ldots,x_d^{\frac{1}{n}}]] $.
	But the extension of the field of fractions of $ \field[[x]] $ to the field of fractions of $ \field [[x_1^{\frac{1}{n}},\ldots,x_d^{\frac{1}{n}}]] $ is Galois and hence all the other roots are also contained in $ \field [[x_1^{\frac{1}{n}},\ldots,x_d^{\frac{1}{n}}]] $.
\end{proof}

\smallskip 

Now, we can come to the

\begin{proof}[Proof of Proposition \ref{Prop:Remain_Main}]
	We want to prove that $ f \in \field [[ \x ]][z] $ is quasi-ordinary.
	By the assumption that our invariant ends with $ \GOOD $ we obtained an overweight deformation from which we constructed the roots in the previous parts.
	It remains to show that the difference of the different roots are monomials times a unit. 
	
	By \eqref{coeff_c_t}, a root $ \zeta $ is determined by the choice of certain roots of unity $ ( \eta_1, \ldots, \eta_g ) $ which determine the values $ ( c_0, \ldots, c_{ g-1 } ) $.
	Let $ \zeta_1, \zeta_2 $ be two roots of $ f $ and let $ ( \eta_1^{(1)}, \ldots, \eta_g^{(1)} ) $ resp.~$ ( \eta_1^{(2)}, \ldots, \eta_g^{(2)} ) $ be the corresponding roots of unity.
	If $ \eta_1^{(1)} \neq \eta_1^{(2)} $ then it follows from \eqref{coeff_c_t} and 
	Lemma \ref{Lem:struc_u_0} that
	$$
		\zeta_1 - \zeta_2 = x^{\lambda_1} \cdot \epsilon_1,
	$$
	for a unit $ \epsilon_1 $.
	Note that $ \epsilon_1 $ corresponds to a unit in $ \field [[ S]] $.
	(Recall that we have to replace $ s_i $ by $ x_i^{\frac{1}{n}} $ for all $ i $).
	
	Therefore suppose $ \eta_j^{(1)} = \eta_j^{(2)} $, for all $ j \in \{ 1, \ldots, r - 1 \} $, and $ \eta_r^{(1)} \neq \eta_r^{(2)} $, for some $ r \geq 1 $.
	By 
	Lemma \ref{Lem:struc_u_0}, all the terms
	before $ S^{n \lambda_r } $ do only depend on $ c_0, \ldots, c_{r-1} $ (which coincide for both roots by the assumption on the roots of unity and \eqref{coeff_c_t}).
	Hence, these must vanish in $ \zeta_1 - \zeta_2 $.
	Moreover, since $ d_r ( c_0, \ldots, c_r ) $ is linear in $ c_r $ we get
	$$
		\zeta_1 - \zeta_2 = x^{\lambda_r} \cdot \epsilon_r,
	$$
	for some unit $ \epsilon_r $ as before.
	This 
	implies the assertion 
	of the proposition. 
\end{proof}

\begin{Rk}
	As the referee suggested, one can find another proof of Proposition \ref{Prop:Remain_Main} using section 5.4 \cite{GP1}.
	Theorem 3 and Remark 41, loc.~cit., imply that under the hypothesis of the proposition, there are local ring extensions
	$$
		K[[\x]] \subset K[[\x]][u_0]/ \langle f \rangle \subset K[[\x^{L_\gqz}]],
	$$
	where $ L = \IZ^d + \IZ v_1 + \ldots + \IZ v_g $.
	Since the ring extension $ K[[\x]] \subset K[[\x^{L_\gqz}]] $ is defined by a lattice extension, it is unramified over the torus.
	This implies that $ K[[\x]] \subset  K[[\x]][u_0]/ \langle f \rangle $ is also unramified over the torus, thus $ f $ is quasi-ordinary.
\end{Rk}

The proof that we provided
 reveals that the invariant $ \theinvariant(f;\x;z) $ does not only tell us if a given $ f $ is quasi-ordinary, but provides even more information in the \qo case.
More precisely,

\begin{Prop}
	\label{Prop:gen_semi_gp}
Let $ f $ be \qo and irreducible and
$$ 
		\theinvariant(f;\x;z) = ( \, v_1; \, v_2; \, \ldots; \, v_g; \, \GOOD \,) .
$$	
The entries $ ( v_{1}, \ldots, v_{g} ) =(\gamma_1,\ldots,\gamma_g) $ yield a system of generators of the semi-group associated to the \qo singularity.
\end{Prop}

\begin{proof}
	This is a direct application of the proof of Proposition  \ref{sdir}.
\end{proof}

Note that the entries are not necessarily the minimal set of generators of the semigroup,
e.g., if the semigroup associated to the quasi-ordinary projection is not normalized, see \cite{GP2}.
In \cite{LipmanThesis}, Lipman shows that any quasi-ordinary hypersurface can be parametrized by a normalized branch, see also the appendix of \cite{G}.
In connection to this see also the recent paper by Garc\'ia Barroso, Gonz\'alez P\'erez, and Popescu-Pampu \cite{GBGPPP}.

\medskip

Besides the order of a hypersurface at a point another measure for the complexity of a singularity is the so called log canonical threshold.
In \cite{Pedro_et_alt_lct}, Theorem 3.1, it was proved that the log canonical threshold of a \qo singularity can be computed from the generators of the corresponding semi-group.
Therefore we obtain:

\begin{Rk}
	If $ f $ is an irreducible \qo hypersurface singularity and $ ( \x, z ) $ such that the last entry of $ \theinvariant( f; \x ; z) $ is $ \GOOD $,
	then one can compute from $ \theinvariant(f;\x;z) $ the log canonical threshold of $ f $.
\end{Rk}


\smallskip 
	
	Finally, let us remark that our original approach to the characterization was via \cite{ACLM_Newton} (see also \cite{GV}).
	There, \qo singularities are characterized using the Newton process and Newton trees. 
	One can show that there is a one-to-one correspondence between the Newton process and our construction which may yield another proof of our characterization.
	It is worth noticing that the two characterizations are different in nature. This difference may be compared to the difference between resolution of singularities of quasi-ordinary singularities by a sequence of toric maps corresponding to the Puiseux pairs, and their resolution by one toric map, see \cite{GP1}.

%
%
%
%
%
%
%
%
%
%
%
%
%
%
%
%
%

\label{Literature}

\newcounter{zaehler}
\setcounter{zaehler}{1}
\newcommand{\art}{\thezaehler}
 

\end{document}